\documentclass[12pt,twoside]{article}
\usepackage{amsfonts}
\usepackage{amsmath, amsthm, amsfonts, amssymb, color}
\usepackage{mathrsfs}
\usepackage{dsfont}
\renewcommand{\bar}{\overline}
\renewcommand{\hat}{\widehat}
\renewcommand{\tilde}{\widetilde}

\newtheorem{thm}{Theorem}[section]

\newtheorem{lem}[thm]{Lemma}

\theoremstyle{definition}

\newcommand{\scr}[1]{\mathscr #1}
\definecolor{wco}{rgb}{0.5,0.2,0.3}

\numberwithin{equation}{section} 
\newtheorem{rem}{Remark}[section]

\newcommand{\ua}{\uparrow}

\title{{\bf Weak convergence of path-dependent SDEs driven by fractional Brownian motion with irregular coefficients}
}
\author{
{\bf  Yongqiang Suo$^{a)}$, Chenggui Yuan$^{a)}$, Shao-Qin Zhang$^{b)}$
 }\\
\footnotesize{a) Department of Mathematics, Swansea University, Bay campus, SA1 8EN, UK}\\
\footnotesize{b) School of Statistics and Mathematics}\\
 \footnotesize{ Central University of Finance and Economics, Beijing 100081, China}
}

\begin{document}

\def\1{\mathds{1}}
\def\A{\mathscr{A}}
\def\G{\mathscr{G}}
\def\eq{\equation}
\def\bg{\begin}
\def\ep{\epsilon}
\def\x{\|x\|}
\def\y{\|y\|}
\def\xr{\|x\|_r}
\def\xrr{(\sum_{i=1}^T|x_i|^r)^{\frac{1}{r}}}
\def\R{\mathbb R}
\def\ff{\frac}
\def\ss{\sqrt}
\def\B{\mathbf B}
\def\N{\mathbb N}
\def\kk{\kappa} \def\m{{\bf m}}
\def\dd{\delta} \def\DD{\Dd} \def\vv{\varepsilon} \def\rr{\rho}
\def\<{\langle} \def\>{\rangle} \def\GG{\Gamma} \def\gg{\gamma}
  \def\nn{\nabla} \def\pp{\partial} \def\EE{\scr E}
\def\d{\text{\rm{d}}} \def\bb{\beta} \def\aa{\alpha} \def\D{\scr D}
  \def\si{\sigma} \def\ess{\text{\rm{ess}}}\def\lam{\lambda}
\def\beg{\begin} \def\beq{\begin{equation}}  \def\F{\scr F}
\def\Ric{\text{\rm{Ric}}} \def \Hess{\text{\rm{Hess}}}
\def\e{\text{\rm{e}}} \def\ua{\underline a} \def\OO{\Omega}  \def\oo{\omega}
 \def\tt{\tilde} \def\Ric{\text{\rm{Ric}}}
\def\cut{\text{\rm{cut}}} \def\P{\mathbb P} \def\ifn{I_n(f^{\bigotimes n})}
\def\C{\scr C}      \def\alphaa{\mathbf{r}}     \def\r{r}
\def\gap{\text{\rm{gap}}} \def\prr{\pi_{{\bf m},\varrho}}  \def\r{\mathbf r}
\def\Z{\mathbb Z} \def\vrr{\varrho} \def\l{\lambda}
\def\L{\scr L}\def\Tilde{\tilde} \def\TILDE{\tilde}\def\II{\mathbb I}
\def\i{{\rm in}}\def\Sect{{\rm Sect}}\def\E{\mathbb E} \def\H{\mathbb H}
\def\M{\scr M}\def\Q{\mathbb Q} \def\texto{\text{o}} \def\LL{\Lambda}
\def\Rank{{\rm Rank}} \def\B{\scr B} \def\i{{\rm i}} \def\HR{Hat{\R}^d}
\def\to{\rightarrow}\def\l{\ell}\def\ll{\lambda}
\def\8{\infty}\def\ee{\epsilon} \def\Y{\mathbb{Y}} \def\lf{\lfloor}
\def\rf{\rfloor}\def\3{\triangle}\def\H{\mathbb{H}}\def\S{\mathbb{S}}
\def\va{\varphi}

\def\R{\mathbb R}  \def\B{\mathbf
B}
\def\N{\mathbb N} \def\kk{\kappa} \def\m{{\bf m}}
\def\dd{\delta} \def\DD{\Delta} \def\vv{\varepsilon} \def\rr{\rho}
\def\<{\langle} \def\>{\rangle} \def\GG{\Gamma} \def\gg{\gamma}
  \def\nn{\nabla} \def\pp{\partial} \def\EE{\scr E}
\def\d{\text{\rm{d}}} \def\bb{\beta} \def\aa{\alpha} \def\D{\scr D}
  \def\si{\sigma} \def\ess{\text{\rm{ess}}}
\def\beg{\begin} \def\beq{\begin{equation}}  \def\F{\scr F}
\def\Ric{\text{\rm{Ric}}} \def\Hess{\text{\rm{Hess}}}
\def\e{\text{\rm{e}}} \def\ua{\underline a} \def\OO{\Omega}  \def\oo{\omega}
 \def\tt{\tilde} \def\Ric{\text{\rm{Ric}}}
\def\cut{\text{\rm{cut}}} \def\P{\mathbb P} \def\ifn{I_n(f^{\bigotimes n})}
\def\C{\scr C}      \def\aaa{\mathbf{r}}     \def\r{r}
\def\gap{\text{\rm{gap}}} \def\prr{\pi_{{\bf m},\varrho}}  \def\r{\mathbf r}
\def\Z{\mathbb Z} \def\vrr{\varrho} \def\ll{\lambda}
\def\L{\scr L}\def\Tt{\tt} \def\TT{\tt}\def\II{\mathbb I}
\def\i{{\rm in}}\def\Sect{{\rm Sect}}\def\E{\mathbb E} \def\H{\mathbb H}
\def\M{\scr M}\def\Q{\mathbb Q} \def\texto{\text{o}} \def\LL{\Lambda}
\def\Rank{{\rm Rank}} \def\B{\scr B} \def\i{{\rm i}} \def\HR{\hat{\R}^d}
\def\to{\rightarrow}\def\l{\ell}
\def\8{\infty}\def\X{\mathbb{X}}\def\3{\triangle}
\def\V{\mathbb{V}}\def\M{\mathbb{M}}\def\W{\mathbb{W}}\def\Y{\mathbb{Y}}

\def\va{\varphi}
\def\l{\lambda}
\def\var{\varphi}
\renewcommand{\bar}{\overline}
\renewcommand{\hat}{\widehat}
\renewcommand{\tilde}{\widetilde}

\allowdisplaybreaks
\maketitle
\begin{abstract}
In this paper, by using  Girsanov's transformation and the property of the corresponding reference stochastic differential equations, we  investigate weak existence and uniqueness of solutions  and weak convergence of Euler-Maruyama scheme to stochastic functional differential equations with H\"older continuous drift driven by fractional Brownian motion with Hurst index $H\in (1/2,1)$. 
\end{abstract}


\section{Introduction}
The fractional Brownian motion (fBM) appears naturally in modeling stochastic systems with long-range dependence phenomena in applications.
Fractional Brownian motions with Hurst parameter $H\neq1/2$ are neither   Markov processes nor a (weak) semimartingales, 
which makes the study of stochastic differential equations (SDEs) driven by fBMs complicated. The existence and uniqueness of solutions to fractional equations  have received much attention.  \cite{L} obtained existence and uniqueness of solutions to SDEs driven by fBMs with Hurst parameter $H\in(\frac{1}{2},1)$ by using Young integrals (see \cite{Y}) and $p$-variation estimate; \cite{CQ} derived the existence and uniqueness result for $H\in(\frac{1}{4},\frac{1}{2})$ through the same rough-type arguments in \cite{L}; \cite{N} studied SDEs driven by fBMs by using  fractional calculus developed in \cite{Z}.  For more results on existence and uniqueness of solutions to SDEs driven by fBMs, we refere to \cite{BH,H,HN,LL,MS2,NO}  for instance.  Stochastic functional differential equations (SFDEs) are also used to characterise stochastic systems with memory effects.   For the existence and uniqueness of solutions for SFDEs with regular coefficients, one can consult \cite{F,MG,NN}.  In recent years, SDEs driven by fractional Brownian motion with irregular coefficients have received much attention, e.g.\cite{FZ,HN}.  However, for fractional SFDEs with irregular coefficients,  even the weak existence and uniqueness results are not well studied. So, we first study the weak existence and uniqueness for SFDEs driven by fBMs (see Theorem \ref{thm-ex-un} below), based on which we shall give a weak convergence result on the weak solution of path-dependent fractional SDEs with irregular drift (see Theorem \ref{thm1}).  By using the associated Kolmogorov equations, SDEs  with irregular coefficients driven by Brownian motion or L\'evy noise are intensively studied. However, this powerful tool seems hard to be applied to fractional SDEs.  To study weak solutions, we adopt  the Girsanov's transformation.   In the case of SDEs driven by fBMs, it involves fractional calculus  to ensure that the Girsanov's transformation  can be applied, and the related estimates are nontrivial for the irregular path-dependent drift.  

There is a few literature on the convergence of numerical schemes for SDEs driven by fBMs, e.g. \cite{HL, M,MS1,MS,MG,NA,SS}. Recently,  \cite{B}  developed a perturbation argument to investigate the weak convergence of path-dependent SDEs with irregular coefficients by using Girsanov's transformation. Based on our  weak existence and uniqueness result, we  investigate the weak convergence of truncated Euler-Maruyama (EM) scheme for SFDEs driven by fBMs. The drift is path-dependent and irregular, and the exponential integrability of functionals of segment process studied in our work involves fractional calculus, which is more complicated than SFDEs driven by Brownian motion. Explicit convergence order is given for the numerical scheme, and the main ingredient is giving exact estimates for fractional derivatives of functionals of the segment process truncated by  gridpoints, see Lemma \ref{lem1.35}.

The paper is organised as follows: Section 2 is devoted to the Preliminaries containing fractional calculus and fractional Brownian motion; in Section 3, we state our  main results on weak existence and uniqueness  and numerical approximation; proofs are provided in Section 4 and Section 5.

\section{Preliminaries}
\subsection{Fractional integrals and derivatives}
In this subsection, we recall some basic facts about fractional integrals and derivatives, for more details, see \cite{ND,S}.

Let $a,b\in\R$ with $a<b$. For $f\in L^1(a,b)$ and $\alpha>0$, the left-sided fractional Riemann-Liouville integral of order $\alpha$ of $f$  on $[a,b]$ is given by 
\begin{align*}
I_{a+}^\alpha f=\frac{1}{\Gamma(\alpha)}\int_a^x\frac{f(y)}{(x-y)^{1-\alpha}}\d y,
\end{align*}
where $x\in(a,b)$ a.e. $(-1)^{-\alpha}=\e^{-i\alpha\pi}$, $\Gamma$ denotes the Euler function. If $\alpha=n\in\N$, this definition coincides with the $n$-order iterated integrals of $f$. By the definitions, we have the first composition formula
\begin{align*}
I_{a+}^\alpha(I_{a+}^\beta f)=I_{a+}^{\alpha+\beta}f.
\end{align*}
Fractional differentiation may be introduced as an inverse operation. Let $\alpha\in(0,1)$ and $p\ge1$. If $f\in I_{a+}^\alpha(L^p([a,b],\R))$, the function $\phi$ satisfying $f=I_{a+}^\alpha\phi$ is unique in $L^p([a,b],\R)$ and it coincides with the left sided Riemann-Liouville derivative of $f$ of order $\alpha$ given by 
\begin{align*}
D_{a+}^\alpha f(x)=\frac{1}{\Gamma(1-\alpha)}\frac{\d}{\d x}\int_a^x\frac{f(y)}{(x-y)^\alpha}\d y.
\end{align*}
The corresponding Weyl representation reads as follows
\begin{align*}
D_{a+}^\alpha f(x)=\frac{1}{\Gamma(1-\alpha)}\Big(\frac{f(x)}{(x-a)^\alpha}+\alpha\int_a^x\frac{f(x)-f(y)}{(x-y)^{1+\alpha}}\d y\Big),
\end{align*}
where the convergence of the integrals at the singularity $y=x$ holds pointwise for almost all $x$ if $p=1$ and in the $L^p$ sense if $p>1$. By the construction, we have 
\begin{align*}
I_{a+}^\alpha(D_{a+}^\alpha f)=f, \qquad  f\in I_{a+}^\alpha(L^p([a,b],\R)),
\end{align*}
and moreover it holds the second composition formula
\begin{align*}
D_{a+}^\alpha(D_{a+}^\beta f)=D_{a+}^{\alpha+\beta}f,\qquad f\in I_{a+}^{\alpha+\beta}(L^1([a,b],\R)).
\end{align*}

\subsection{ Fractional Brownian motion}
To make the content self-contained, we first recall some basic facts about the stochastic calculus of variations with respect to the fBm with Hurst parameter $H\in(\frac{1}{2},1)$. We refer the reader to \cite{D} for further details.

Fixe $T>0$.  The $d$-dimensional fBm $B^H=\{B^H(t),t\in[0,T]\}$ with Hurst parameter $H$ on the probability space $(\Omega,\mathscr{F},\P)$ can be defined as the centered  Gauss  process  with covariance function 
\begin{align*}
\E(B_t^HB_s^H)=R_H(t,s)=\frac{1}{2}(t^{2H}+s^{2H}-|t-s|^{2H}).
\end{align*}
In particular, if $H=\frac{1}{2}$, $B^H$ is a Brownian motion. Besides, 
\begin{align*}
\E|B_t^H-B_s^H|^p=\E|B_{t-s}^H|^p=|t-s|^{pH}\E|B_1^H|^p\le C(p)|t-s|^{pH},~ p\ge1.
\end{align*}
Then it follows from the Kolmogorov continuity theorem that $B^H$ has $\beta$-H\"older continuous paths, where $\beta\in(0,H)$.
For each $t\in[0,T]$, we denote by $\mathscr{F}_t$ the $\sigma$-algebra generated by  $\{B_s^H:s\in[0,t]\}$ and the $\P$-null sets.

We denote by $\mathscr{E}$ the set of step functions on $[0,T]$. Let $\mathscr{H}$ be the Hilbert space defined as the closure of $\mathscr{E}$ with respect to the scalar product 
\begin{align*}
\langle(I_{[0,t_1]},\cdots,I_{[0,t_d]}),(I_{[0,s_1]},\cdots,I_{[0,s_d]})\rangle_{\mathscr{H}}=\sum_{i=1}^dR_H(t_i,s_i).
\end{align*}
The mapping $I_{[0,t_1]\times\cdots\times I_{[0,t_d]}}\mapsto (B_{t_1}^{H,1},\cdots,B_{t_d}^{H,d})$ can be extended to an isometry between $\mathscr{H}$ and the Gauss space $\mathscr{H}_1$ spanned by $B^H$. Denote this isometry by $\phi\mapsto B^H(\phi)$.
On the other hand, from \cite{D}, we know the covariance kernel $R_H(t,s)$ can be written as 
\begin{align*}
R_H(t,s)=\int_0^{t\wedge s}K_H(t,r)K_H(s,r)\d r,
\end{align*} 
where $K_H$ is a square integrable kernel given by 
\begin{align*}
K_H(t,s)=\Gamma(H+\frac{1}{2})^{-1}(t-s)^{H-\frac{1}{2}}F(H-\frac{1}{2},\frac{1}{2}-H,H+\frac{1}{2},1-\frac{t}{s}),
\end{align*}    
in which $F(\cdot,\cdot,\cdot,\cdot)$ is the Gauss's hypergeometric function (see\cite{D}).

Define the linear operator $K_H^*:\mathscr{E}\rightarrow L^2([0,T],\R^d)$ as follows
\begin{align*}
(K_H^*\phi)(s)=K_H(T,s)\phi(s)+\int_s^T(\phi(r)-\phi(s))\frac{\partial K_H}{\partial r}(r,s)\d r.
\end{align*}
Reformulating the above equality as follows:
\begin{align*}
(K_H^*\phi)(s)=\int_s^T\phi(r)\frac{\partial K_H}{\partial r}(r,s)\d r.
\end{align*}
It can be shown that for all $\phi,\psi\in\mathscr{E}$,  
\begin{align*}
\langle K_H^*\phi,K_H^*\psi\rangle_{L^2([0,T],\R^d)}=\langle\phi,\psi\rangle_{\mathscr{H}},
\end{align*}
and therefore $K_H^*$ is an isometry between $\mathscr{H}$ and $L^2([0,T],\R^d)$. Consequently,  $B^H$ has the following integral representation 
\begin{align*}
B^H(t)=\int_0^tK_H(t,s)\d B(s),
\end{align*}
where $\{B(t):=B^H((K_H^*)^{-1}I_{[0,t]})\}$ is a standard Brownian motion.

According to \cite{D}, the operator $K_H:L^2([0,T],\R^d)\rightarrow I_{0+}^{H+\frac{1}{2}}(L^2([0,T],\R^d))$ associated with the kernel $K_H(\cdot,\cdot)$ is defined as follows
\begin{align}\label{eqk}
(K_H f^i)(t)=\int_0^tK_H(t,s)f^i(s)\d s,~~i=1,\cdots,d.
\end{align}
It can be proved that $K_H$ is an isomorphism and moreover, for each $f\in L^2([0,T],\R^d)$, 
\begin{align}\label{eqf}
(K_Hf)(s)=I_{0+}^1s^{H-1/2}I_{0+}^{H-1/2}s^{1/2-H}f, H>\frac{1}{2}.
\end{align}
Consequently, for each $h\in I_{0+}^{H+1/2}(L^2([0,T],\R^d))$, the inverse operator $K_H^{-1}$ is of the form
\begin{align}\label{eqd}
(K_H^{-1}h)(s)=s^{H-1/2}D_{0+}^{H-1/2}s^{1/2-H}h',~H>\frac{1}{2}.
\end{align}

We conclude this section by  introducing the following Fernique-type lemma  (see \cite{MWX, SB}) and some notation  for future use.
 \begin{lem}\label{F}
 Let $T>0,1/2<\beta<H<1$. Then for any $\alpha<\frac 1 {2T}$, 
 $$\E\exp\{\alpha \|B^H\|_{0,T,\infty}^2\}<\infty,$$
and for any  $\alpha<1/(128(2T)^{2(H-\beta)})$,
 \begin{align*}
 \E[\exp(\alpha\|B^H\|_{0,T,\beta}^2)]\le(1-128\alpha(2T)^{2(H-\beta)})^{-1/2}.
 \end{align*}
 Moreover, we have the following moment estimate for any $k\ge1$:
 \begin{align*}
 \E(\|B^H\|_{0,T,\beta}^{2k})\le 32^k(2T)^{2p(H-\beta)}\frac{(2k)!}{k!}.
 \end{align*}
 \end{lem}

For any $\alpha\in(0,1)$, let $C^\alpha(a,b)$ be the space of $\alpha$-H\"older continuous functions $f$ on the interval $[a,b]$ and set
 \begin{align*}
 \|f\|_{a,b,\alpha}:=\sup_{a\le s\le t\le b}\frac{|f(t)-f(s)|}{|t-s|^\alpha}.
 \end{align*}
Besides, for any continuous function $f\in C([a,b];\R^d)$, let
\begin{align*}
\|f\|_{a,b,\8}=\sup_{a\le s\le b}|f(s)|.
\end{align*}
When $a=0,b=T$, we will simply write $\|f\|_\alpha, \|f\|_\8$ for $\|f\|_{0,T,\alpha}, \|f\|_{0,T,\8}$, respectively.

\section{Main results}
Let $(\R^d,\langle\cdot,\cdot\rangle,|\cdot|)$ be the $d$-dimensional Euclidean space with the inner product $\langle\cdot,\cdot\rangle$ which induces the norm $|\cdot|$.
Let $\R^d\otimes\R^m$ be the set of all $d\times m$-matrices. 
Let $\tau>0$ be a fixed number and $\mathscr{C}=C([-\tau,0];\R^d)$, which is endowed with the uniform norm $\|f\|_\8:=\sup_{-\tau\le\theta\le0}|f(\theta)|$. For $f\in C([-\tau,\8);\R^d)$ and fixed $t>0$, let $f_t\in\mathscr{C}$ be defined by $f_t(\theta)=f(t+\theta),\theta\in[-\tau,0]$. For $a\ge0, [a]$ stipilates the integer part of $a$. Let $\mathscr{B}_b(\R^d)$ be the collection of all bounded measurable functions.  

In this paper, for $H\in (\frac{1}{2},1)$, we consider the following equation:
\begin{align}\label{eq1.1}
\d X(t)=\{b(X(t))+\si Z(X_t)\}\d t+\sigma\d B^H(t), t>0,
\end{align}
with the initial datum $X_0=\xi\in\mathscr{C}$,
where  $\sigma\in\R^d\otimes\R^m$,  $b:\R^d\rightarrow\R^d,d\ge m$ and $Z:\mathscr{C}\rightarrow\R^m$ are measurable, $B^H(t)$ is an $m$-dimensional fBM on the probability space $(\Omega,\mathscr{F},(\mathscr{F}_t)_{t\ge0},\P)$.  Consider a reference SDE as follows:
\begin{align}\label{eq1.6}
\d Y(t)=b(Y(t))\d t+\sigma\d B^H(t),~~t>0, Y(0)\in\R^d.
\end{align}
Let $\xi\in \mathscr{C}$, and let $Y^{\xi(0)}(\cdot)$ be a solution of \eqref{eq1.6} with $Y^{\xi(0)} (0)=\xi(0)$. We  extend $Y^{\xi(0)}(\cdot)$ from $[0,\8)$ to $[-\tau,\8)$ in the following way:
\begin{align}\label{eq1.7}
Y^\xi(t)=\xi(t)I_{[-\tau,0)}(t)+Y^{\xi(0)}(t)I_{[0,\8)}(t), t\in[-\tau,\8),\xi\in\mathscr{C}.
\end{align} 
Then the weak existence and uniqueness of solutions to \eqref{eq1.1}  and the weak convergence of EM scheme will be studied by using Girsanov's transform and the extended solutions to the reference equation \eqref{eq1.6}.

We first introduce the following assumptions on $b$ and $Z$ for the  weak existence and uniqueness result.
\begin{enumerate}
\item [(A1)]  There exists  a constant $K_1\in\R$ such that
$$\<b(x)-b(y),x-y\>\leq K_1|x-y|^2.$$

\item [(A2)] There exist $C_1>0$ and $q_0\geq 0$ such that $|b(x)|\leq C_1(1+|x|^{q_0})$.

\item [(A3)] There exist  $\alpha\in (H-1/2,1]$, $p>0$, $C_2>0$ and $C_3\geq 0$ such that 
\begin{align}\label{equ-hh}
\left|Z(\eta_1)-Z(\eta_2)\right|&\leq C_2\|\eta_1-\eta_2\|^\alpha_\infty \left(1+\|\eta_1\|_\infty^p+\|\eta_2\|_\infty^p\right),\\
\<\sigma Z(\eta_1+\eta_2),\eta_1(0)\>&\leq C_3\left(1+\|\eta_2\|^{q_1}_\infty+\|\eta_1\|^2_{\infty}\right),~\eta_1,\eta_2\in\mathscr{C}.\label{equ-hh-1}
\end{align}

\end{enumerate}
Our result on existence and uniqueness of weak solutions to \eqref{eq1.1} is the following theorem. 
\begin{thm}\label{thm-ex-un}
Assume (A1)-(A3). For any $\xi \in\mathscr{C}$ with  $\theta\in (\frac {2H-1} {2\alpha},1] $ and $\bar C_1>0$  such that 
\begin{align}\label{ineq-ini}
\left|\xi(r)-\xi(s)\right|\leq \bar C_1|r-s|^\theta,~-\tau\leq r\leq s\leq 0,
\end{align}
then the equation \eqref{eq1.1} has a unique weak solution with $X_0=\xi$.

\end{thm}

\begin{rem}
The condition  \eqref{ineq-ini} is for us to  use Girsanov transformation to remove the drift term $ Z(\cdot)$ of equation \eqref{eq1.1}. Given $T>0$. For any $\gamma\in C([-\tau,T],\R^d)$ with $\gamma_0=\xi_0$, to ensure that  $\{\int_0^s Z(\gamma_r)\d r\}_{s\in [0,T]}$ belongs to the Cameron-Martin space of the fBM, it is necessary that the integral $\int_0^\cdot Z(\gamma_s)\d s\in I_{0+}^{H+\frac{1}{2}}(L^2([0,T],\R^d))$. This means we need  $Z(\gamma_\cdot)\in I_{0+}^{H-\frac{1}{2}}(L^2([0,T],\R^d))$.  Note that for $t\in [0,T\wedge \tau]$, we have
\begin{align*}
\|\gamma_\cdot\|_{0,t,\alpha}=\sup_{0\leq r\leq u\leq t}\frac {\|\gamma_u-\gamma_r\|_\infty} {|u-r|^\alpha}=\sup_{0\leq r\leq u\leq t,v\in [-\tau,0]}\frac {|\gamma(u+v)-\gamma(r+v)|} {(u-r)^{\alpha}}\geq \|\xi\|_{-\tau,0,\alpha}.
\end{align*} 
Hence, despite imposing regularity conditions on $Z$, we also need an additional assumption on the initial value $\xi$. If $Z$ is $\alpha$-H\"older continuous and $\xi$ is $\theta$-H\"older continuous, then our conditions on  $\xi$ yields that $\theta\alpha>H-\frac 1 2$, which ensure that $\{\int_0^s Z(\gamma_r)\d r\}_{s\in [0,T]}$ is in the Cameron-Martin space. 
\end{rem}

Next, we shall study the weak convergence of the numerical approximation to \eqref{eq1.1}. In \eqref{eq1.1}, $\sigma$ is a $d\times m$ matrix with $d\geq m$. For $d>m$, this equation is obviously degenerate. Hence, we shall introduce the pseudo-inverse of $\sigma$ to  cover some degenerate models, such as stochastic Hamilton systems.  Denote by $\mathrm{Ran}(\sigma)$ the range of $\sigma$, i.e. $\mathrm{Ran}(\sigma)=\sigma(\R^m)$.  If $\mathrm{Ran}(\sigma)$ contains nonzero vectors, then $\sigma\sigma^*$ is a bijective from $\mathrm{Ran}(\sigma)$ onto $\mathrm{Ran}(\sigma)$, whose inverse is denoted by $(\sigma\sigma^*)^{-1}\Big|_{\mathrm{Ran}(\sigma)}$. Let $\pi_*$ be the orthogonal projection from $\R^d$ to $\mathrm{Ran}(\sigma)$. Then $\R^d$ has the following decomposition:
$$\R^d=\pi_*\R^d\oplus (I_{d\times d}-\pi_*)\R^d\equiv\mathrm{Ran}(\sigma)\oplus (I_{d\times d}-\pi_*)\R^d,$$
where $I_{d\times d}$ is the identity matrix of $\R^d$.  We define  $\hat{\sigma}^{-1}$,  the pseudo-inverse of $\sigma$, as follows
\beg{align*}
\hat\si^{-1} v=\sigma^*\left(\left(\sigma\sigma^*\right)^{-1}\Big|_{\mathrm{Ran}(\sigma)}\pi_*v\right),~v\in\R^d.
\end{align*}
Then $\|\hat{\sigma}^{-1}\|=\left\|(\sigma\sigma^*)^{-1}\Big|_{\mathrm{Ran}(\sigma)}\right\|$. 
In particular, if $\sigma$ is of the form $\left(\begin{array}{c} 0 \\ \sigma_0\end{array}\right)$ with $\sigma_0$ is an invertible $m\times m$-matrix and $0$ is a $(d-m)\times m$ zero matrix, then 
$$\hat{\sigma}^{-1}=\left( 0^* ,  \sigma_0^{-1} \right),\qquad \|\hat{\sigma}^{-1}\|=\|\sigma_0^{-1}\|.$$ 

We need stronger assumptions on $b$ and $Z$ for numerical approximation.
\begin{enumerate}
\item[(H1)]  (A1) holds and there exists a constant $L_1>0$ such that
\begin{align}\label{eq1.2}
|b(x)-b(y)|\le L_1|x-y|,~x,y\in\R^d.
\end{align}
Moreover, if $\mathrm{Ran}(\sigma)\neq \R^d$,  we also assume that there exist a matrix $A$ on $ (I_{d\times d}-\pi_*)(\R^d)$  and a measurable function $b_*: \mathrm{Ran}(\sigma)\rightarrow (I_{d\times d}-\pi_*)(\R^d) $ such that
$$(I_{d\times d}-\pi_*)b(x)=A(I_{d\times d}-\pi_*) x+b_*(\pi_* x), ~x\in\R^d.$$
\item[(H2)] $Z$ is   H\"older continuous with the exponent $\alpha\in(1-\frac{1}{2H},1]$, that is 
\begin{align}\label{eq1.3}
|Z(\xi)-Z(\eta)|\le L_2\|\xi-\eta\|_{\8}^\alpha, \xi,\eta\in\mathscr{C};
\end{align}
\item[(H3)] the initial value $\xi\in\mathscr{C}$ is H\"older continuous with exponent $\theta\in (\frac{2H-1}{2\alpha}, 1]$, that is,
\begin{align}\label{eq1.4}
|\xi(t)-\xi(s)|\le L_3|t-s|^\theta,~~s,t\in[-\tau,0].
\end{align}
\end{enumerate}
By these conditions, it follows from Theorem \ref{thm-ex-un}  that \eqref{eq1.1} has a unique weak solution with $X_0=\xi$.

\begin{rem}
Since the  pseudo-inverse of $\sigma$  is the inverse of $\sigma$ if it is invertible, our setting can  unify non-degenerate and some degenerate models. A typical example for the equation with $\{0\}\subsetneq\mathrm{Ran}(\sigma)\subsetneq \R^d$ is the following  stochastic Hamiltonian system ($d=2m$):
\begin{align*}
\left\{\begin{array}{l}
\d X^{(1)}(t)=X^{(2)}(t)\d t\\
\d X^{(2)}(t)=b_0(X^{(1)}(t),X^{(2)}(t))\d t+ Z_0(X^{(1)}_t,X^{(2)}_t)\d t+\sigma_0\d B^H(t),
\end{array}
\right.
\end{align*} 
where  $\sigma_0$ is an invertible  $m\times m$-matrix. For any $\eta_1,\eta_2\in\mathscr{C}, x=(x^{(1)},x^{(2)})\in\R^{2m}$, we set 
\beg{align*}
b(x^{(1)},x^{(2)})=\left(\begin{array}{c} x^{(2)}\\b_0(x^{(1)},x^{(2)})\end{array}\right), \quad Z(\eta_1,\eta_2)=\left(\begin{array}{c}0\\ \sigma_0^{-1}Z(\eta_1,\eta_2)\end{array}\right),\quad \sigma=\left(\begin{array}{c} 0 \\ \sigma_0\end{array}\right),
\end{align*}
Then 
\begin{align*}
\d X(t)\equiv \d\left(\begin{array}{c} X^{(1)}(t)\\ X^{(2)}(t)\end{array}\right)=\left(b(X(t))+\sigma Z(X_t)\right)\d t+\sigma\d  B^H(t),
\end{align*}
and in this case, $\pi_*(x^{(1)},x^{(2)})=(0,x^{(2)})$, $b_*((0,x^{(2)}))=(x^{(2)},0)$ and $A\equiv 0$ in (A1).

\end{rem}

We can construct the  EM scheme now. Let $\delta\in(0,1)$ be the step-size given by $\delta=\tau/M$ for some $M\in \N$ sufficiently large. The continuous time EM scheme associated with \eqref{eq1.1} is defined as below:
\begin{align}\label{eq1.5}
\d X^{(\delta)}(t)=\{(I-\pi_*)b(X^{(\delta)}(t))+\pi_*b(X^{(\delta)}(t_\delta))+\si Z(\hat{X}_t^{(\delta)})\}\d t+\sigma\d B^H(t),~~t>0
\end{align}
with the initial value $X^{(\delta)}(u)=X(u)=\xi(u), u\in[-\tau,0]$, where  $t_\delta:=[t/\delta]\delta$ and $\hat{X}_t^{(\delta)}\in\mathscr{C}$ defined as follows 
$$\hat{X}_t^{(\delta)}(u)=X^{(\delta)}((t+u)\wedge t_\delta),~ u\in[-\tau,0].$$
For $t\in [0,\delta)$, $\hat{X}_t^{(\delta)}(u)=X^{(\delta)}((t+u)\wedge 0)=\xi((t+u)\wedge 0)$, and 
$$H(t):=\pi_* X^{(\delta)}(t)=\pi_* X^{(\delta)}(0)+ \pi_*b(X^{(\delta)}(0))t+\int_0^t\sigma Z(\hat X_s^{(\delta)})\d s+\sigma B^{H}(t).$$
Then it follows from (A1) that 
\beg{align*}
(I_{d\times d}-\pi_*)X^{(\delta)}(t)&=(I_{d\times d}-\pi_*)X^{(\delta)}(0)+\int_0^t(I_{d\times d}-\pi_*)b(\pi_*X^{(\delta)}(s)+(I_{d\times d}-\pi_*)X^{(\delta)}(s))\d s\\
&=(I_{d\times d}-\pi_*)X^{(\delta)}(0)+\int_0^tA(I_{d\times d}-\pi_*) X^{(\delta)}(s)\d s+\int_0^t b_*(H(s))\d s
\end{align*}
which implies that 
\begin{align*}
(I_{d\times d}-\pi_*)X^{(\delta)}(t)=\e^{At}(I_{d\times d}-\pi_*)X^{(\delta)}(0)+\int_0^t \e^{A(t-s)}b_*(H(s))\d s.
\end{align*}
Thus, $X^{(\delta)}(t)=(I_{d\times d}-\pi_*)X^{(\delta)}(t)+\pi_* X^{(\delta)}(t)$ can be obtained explicitly on $[0,\delta]$. By induction,  we can get $X^{(\delta)}(t)$ explicitly.

Let
$$\bar K_1=2K_1+\1_{[K_1\geq 0]}+\frac {|K_1|} 2\1_{[K_1<0]},\qquad \bar K_2=\1_{[K_1\geq 0]}+\frac 2 {|K_1|}\1_{[K_1<0]},$$
and  
$$\Phi(\bar K_1,\bar K_2,T)=\sqrt{\frac {\bar K_2\left( \e^{\bar K_1 T}-1\right)} {\bar K_1}}.$$
Our main result on the weak convergence of EM scheme to \eqref{eq1.1}   is stated as follows. 
\begin{thm}\label{thm1}
Assume (H1)-(H3)  and $\mathrm{Ran}(\sigma)\neq \{0\}$. For $T>0$ and $\delta>0$ such that
\begin{align}\label{add-fZ}
& \frac {2L_2^2T^{2-2H}(1+(H-1/2)^2C_0^2)} {(1-H)\Gamma^2(3/2-H)} \|\sigma\|^2\1_{[\alpha=1]} <\left(C_1\Phi(\bar K_1,\bar K_2,T)+1\right)^{-2} \frac 1 {2T},
\end{align}
and
\begin{align}\label{eq2.5}
&\frac{16L_1^2\|\sigma\|^2\|\hat{\sigma}^{-1}\|^2T^{2\beta}}{\Gamma^2({\frac{3}{2}-H)}(1-H)}\left(L_1\Phi(\bar K_1,\bar K_2,T)+1\right)^2\bigg\{[1+C_0(H-\frac{1}{2})]^2T^{2-2H}\delta^{2\beta}\nonumber\\
&~~+8\delta^{2\beta+1-2H}(T-\delta)^{2-2H}T^{2H-1}\Big[
\frac{1}{(1+2\beta-2H)^2}+\frac{16^H}{(3-2H)^2}+\frac{2^{2H+1}}{(2H-1)^2}
\Big]\nonumber\\
&~~+\frac{4(1-H)\delta^{2\beta+4-4H}}{(\beta+2-2H)^2}\bigg\}\nonumber\\
&~~
+\left\{\frac{8L_2^2\|\sigma\|^{2\alpha}[1+C_0(H-\frac{1}{2})]^2T^{2(\alpha\beta+1-H)}}{\Gamma^2(\frac{3}{2}-H)(1-H)}\left(L_1\Phi(\bar K_1,\bar K_2,T)+1\right)^{2\alpha}\right.\nonumber\\
&~~+\frac{48L_2^2\|\sigma\|^{2\alpha}}{\Gamma^2(\frac{3}{2}-H)}\left(TL_1\left(L_1\Phi(\bar K_1,\bar K_2,T)+1\right)+1\right)^{2\alpha}\nonumber\\
&~~\times\Big[\frac{\mathcal{B}^2(\frac{3}{2}-H,\alpha{(\beta\wedge\theta)}+1/2-H)T^{2\alpha\beta+3-4H}}{2\alpha{(\beta\wedge\theta)}+3-4H}\nonumber\\
&~~\left.+\frac{\delta^{2\alpha{(\beta\wedge\theta)}+1-2H}T^{2-2H}}{1-H}
\Big(\frac{16^H}{(3-2H)^2}+\frac{2^{2H+1}}{(2H-1)^2}
\Big)\Big]\right\}\1_{[\alpha=1]}<1/(128(2T)^{2(H-\beta)}),
\end{align}
where $\beta\in (\frac {2H-1} {2\alpha},H)$ and
$
C_0=\int_0^1\frac{u^{\frac{1}{2}-H}-1}{(1-u)^{\frac{1}{2}+H}}\d u.
$
Then for any bounded measurable function $f$ on $\R^d$ there exists a constant $C_{T}$ such that
\begin{align}\label{eq2.1}
|\E f(X(t))-\E f(X^{(\delta)}(t))|\le C_T\delta^{\alpha{(\beta\wedge\theta)}+\frac{1}{2}-H}.
\end{align}

\end{thm}

\section{Proof of Theorem \ref{thm-ex-un}}

We first introduce the following lemma on the existence and uniqueness of solutions to \eqref{eq1.6}.
\begin{lem}\label{lem-mom}
Assume (A1). Then \eqref{eq1.6} has a unique strong solution and 
\beg{align}\label{ineq-Y-mon} 
\left|Y(t)\right|\leq \e^{\frac {\bar K_2 t} 2}|Y(0)|+\sqrt{\bar K_2}\left(\int_0^t e^{\bar K_1(t-r)}\left|b(\sigma B^H(r))\right|^2\d r\right)^{\frac 1 2}+|\sigma \left(B^H(t)\right)|,~t\geq 0.
\end{align}
Furthermore, if (A2) holds, then 
$$\E\|Y\|_{0,t,\beta}^q<\infty,~q>0,t>0,0<\beta<H.$$
\end{lem}
\begin{proof}
(1) Let $U(t)=Y(t)-\sigma B^H(t)$. Then $U(t)$ satisfies 
\begin{align}\label{equ-U}
\d U(t)=b(U(t)+\sigma B^H(t))\d t,~ U(0)=Y(0).
\end{align}
Set $\bar b(u,t)=b(u+\sigma B^H(t))$. Then it is easy to see that
$$\<\bar b(u_1,t)-\bar b(u_2,t), u_1-u_2\>\leq K_1|u_1-u_2|^2,$$
which implies that \eqref{equ-U} has a unique solution. Moreover, it follows from the chain rule and the H\"older inequality that
\beg{align*}
\d |U(t)|^2&=2\<\bar b(U(t),t),U(t)\>\d t\\
&\leq 2K_1|U(t)|^2+2\<b(\sigma B^H(t)),U(t)\>\d t\\
&\leq \bar K_1 |U(t)|^2\d t+\bar K_2\left|b(\sigma B^H(t))\right|^2\d t.
\end{align*}
Then for any $t\geq s$
\beg{align*}
\left|Y(t)\right|&\leq \e^{\frac 1 2 (t-s)\bar K_1}|Y(s)|+\sqrt{\bar K_2}\left(\int_s^t e^{\bar K_1(t-r)}\left|b(\sigma B^H(r))\right|^2\d r\right)^{\frac 1 2}\\
&\qquad +\e^{\frac {\bar K_1 (t-s)} 2}|\sigma B^H(s)|+|\sigma B^H(t)|,
\end{align*}
which implies our first claim. 

(2) For any $0<\beta<H$,
\begin{align*}
\frac {|Y(t)-Y(s)|} {(t-s)^{\beta}}&\leq \frac 1 {(t-s)^\beta}\int_s^t|b(Y(r))|\d r+\|\si\|\|B^H\|_{0,t,\beta}\\
&\leq \frac {C_1} {(t-s)^\beta}\int_s^t\left(1+|Y(r)|^{q_0}\right)\d r+\|\si\|\|B^H\|_{0,t,\beta}\\
&\leq C_1(t-s)^{1-\beta}+\|\si\|\|B^H\|_{0,t,\beta}+3^{(q_0-1)^+}(t-s)^{1-\beta}\e^{\frac {t\bar K_1^+ q_0} 2} |Y(0)|\\
&\qquad + C_1 3^{(q_0-1)^+ } \left(\sqrt {\bar K_2}C_1\left(1+\|\sigma\|^{q_0}\|B^H\|_{0,t,\infty}^{q_0}\right)\right)^{q_0}\e^{\frac {q_0K_1^+t} 2}t^{q_0}(t-s)^{1-\beta}\\
&\qquad +C_1 3^{(q_0-1)^+  }   \left(\|\sigma\|\|B^H\|_{0,t,\infty}\right)^{q_0} (t-s)^{t-\beta}, 
\end{align*}
which yields 
\begin{align}\label{add-ineq-beY}
&\|Y\|_{0,t,\beta}\leq C_1t^{1-\beta}+\|\si\|\|B^H\|_{0,t,\beta}+C_13^{(q_0-1)^+}t^{1-\beta}\|\si\|^{q_0}\|B^H\|^{q_0}_{0,t,\infty}\nonumber \\
&\qquad  +C_13^{(q_0-1)^+}t^{1-\beta}\e^{\frac {t\bar K_1^+ q_0} 2} \left(\|Y\|_{0,t,\infty}+  t^{q_0}\left(\sqrt {\bar K_2}C_1\left(1+\|\sigma\|^{q_0}\|B^H\|_{0,t,\infty}^{q_0}\right)\right)^{q_0}\right).
\end{align}
Combining this with \eqref{ineq-Y-mon}, it is clear that our second claim holds.
 
\end{proof}

Next lemma is to investigate the exponential martingale, which is crucial to prove Theorem \ref{thm-ex-un}. Fix any $T>0$. Let 
\begin{align*}
\left\{\tilde B^H(t)\right\}_{t\in [0,T]}&=\left\{B^H(t)-\int_0^t Z(Y_s^\xi)\d s\right\}_{t\in [0,T]},\\
R^\xi(t)&=\exp\left(\int_0^t\left\<K_H^{-1}\left(\int_0^\cdot Z( Y_r^\xi)\d r\right)(s),\d  B_s\right\>\right.\\
&\qquad\qquad \left. -\frac{1}{2}\int_0^t\left|K_H^{-1}\left(\int_0^\cdot Z( Y_r^\xi)\d r\right)\right|^2(s)\d s\right), t\in[0,T].
\end{align*}  
\begin{lem}\label{lem-Eexph}
Let the assumptions of Theorem \ref{thm-ex-un} hold. Then
\begin{enumerate} 
\item[$(1)$] $\left\{\tilde B^H(t)\right\}_{t\in [0,T]}$ is a fractional Brownian motion under $R^\xi(T)\P$. 
\item[$(2)$] Assume in addition that $q_0=1$ in (A2). If  there exist $C_4\geq 0$, $C_5\geq 0$  and $p\in (0,1)$ such that
\begin{align}\label{equ-hh-2}
|Z(\eta_1)-Z(\eta_2)|\leq C_4\|\eta_1-\eta_2\|^{\alpha}_\infty\wedge \left(1+C_5(\|\eta_1\|_\infty^p+\|\eta_2\|_\infty^p)\right),
\end{align}
then for any $C\geq 0$
\beg{align}\label{Eexp}
\E\exp\left\{C\int_0^T\left|K_H^{-1}\left(\int_0^\cdot Z(Y^\xi_r)\d r\right)(s)\right|^2\d s\right\}<\infty.
\end{align}
\item[$(3)$] If \eqref{equ-hh-2} holds with $p=1$ and $T>0$ is small enough such that
\begin{align}\label{ineq-on-T}
&\left(\frac {2C_4^2T^{2-2H}(1+(H-1/2)^2C_0^2)} {(1-H)\Gamma^2(3/2-H)}\mathbf{1}_{[\alpha=1]}+\frac {16C_4^2C_5^2T^{2-2T}(H-1/2)^2} {(1-H)\Gamma^2(3/2-H)}\right)\|\sigma\|^2\nonumber\\
&\qquad <\left(C_1\Phi(\bar K_1,\bar K_2,T)+1\right)^{-2} \frac 1 {2T}
\end{align}
where $C_0$ is defined in Theorem \ref{thm1}, then \eqref{Eexp} holds for some $C>1$. 
\end{enumerate} 
\end{lem}
\begin{proof}
If \eqref{ineq-ini} holds for $\theta\geq H$, then \eqref{ineq-ini} holds for $\theta\in (H-1/2,H)$. Hence, we shall assume that $\theta\in (H-1/2,H)$ in the following proof.

(1) It follows from \eqref{eqd} that 
\begin{align}\label{ineq-KHZ}
K_H^{-1}\left(\int_0^\cdot Z(Y^\xi_r)\d r\right)(s)
&=s^{H-\frac{1}{2}}D_{0+}^{H-\frac{1}{2}}\left(\cdot^{\frac{1}{2}-H} Z( Y^\xi_\cdot)\right)(s)\nonumber\\
& =\frac{H-\frac{1}{2}}{\Gamma(\frac{3}{2}-H)}\Bigg[\frac{s^{\frac{1}{2}-H}}{H-\frac{1}{2}} Z( Y^\xi_s)
+s^{H-\frac{1}{2}} Z( Y^\xi_s)\int_0^s\frac{s^{\frac{1}{2}-H}-r^{\frac{1}{2}-H}}{(s-r)^{\frac{1}{2}+H}}\d r\nonumber\\
& \qquad+s^{H-\frac{1}{2}}\int_0^s\frac{ Z(Y^\xi_s)-Z(Y^\xi_r)}{(s-r)^{\frac{1}{2}+H}}r^{\frac{1}{2}-H}\d r\Bigg]\nonumber\\
& =:\frac{H-\frac{1}{2}}{\Gamma(\frac{3}{2}-H)}(J_1(s)+J_2(s)+J_3(s)).
\end{align}
By \eqref{equ-hh}, we have
\begin{align}\label{add-ineq-J1}
\left|J_1(s)\right|^2&\leq \frac {s^{1-2H}} {(H-\frac 1 2)^2}\left(C_2\left(1+\|Y^\xi_s\|_\infty^p\right)\|Y^\xi_s\|_\infty^\alpha+|Z(0)|\right)^2\nonumber\\
&\leq Cs^{1-2H}\left(1+\|Y^\xi\|_{-\tau,s,\infty}^{2(p+\alpha)}\right).
\end{align}
\begin{align}\label{add-ineq-J3}
\left|J_3(s)\right|&\leq s^{H-\frac 1 2}\int_0^s\frac {\|Y^\xi_s-Y^\xi_r\|_\infty^\alpha\left(1+\|Y^\xi_s\|_\infty^p+\|Y^\xi_r\|_\infty^p\right)} {(s-r)^{H+1/2} r^{H-1/2}}\d r \nonumber\\
&\leq 2s^{H-\frac 1 2}\left(1+\|Y^\xi\|_{-\tau,s,\infty}^p\right)\int_0^s \frac {\|Y^\xi\|^\alpha_{-\tau,r,\theta}(s-r)^{\theta\alpha}} {(s-r)^{H+1/2}r^{H-1/2}}\d r \nonumber\\
&=2s^{\theta\alpha+\frac 1 2-H}\mathcal{B}\left(\frac 3 2-H,\theta\alpha+\frac 1 2-H\right)\left(1+\|Y^\xi\|_{-\tau,s,\infty}^p\right)\|Y^\xi\|^\alpha_{-\tau,s,\theta}, 
\end{align}
where $\mathcal{B}$ is Beta function.
For  $J_2$, we have
\beg{align}\label{add-ineq-J2}
|J_2(s)|^2&\leq s^{2H-1}\left|\int_0^s\frac{s^{\frac{1}{2}-H}-r^{\frac{1}{2}-H}}{(s-r)^{\frac{1}{2}+H}}\d r\right|^2\left(C_2\left(1+\|Y^\xi_s\|_\infty^p\right)\|Y^\xi_s\|_\infty^\alpha+|Z(0)|\right)^2\nonumber\\
&\leq Cs^{1- 2H}\left(1+\|Y^\xi\|_{-\tau,s,\infty}^{2(p+\alpha)}\right)^2.
\end{align}
  Then it is clear that 
$$\int_0^t\left|K_H^{-1}\left(\int_0^\cdot Z(Y^\xi_r)\d r\right)(s)\right|^2\d s<\infty, t\in [0,T], \P\mbox{-a.s.}$$

Let
$$\tau_n=\inf\left\{t>0~\Big|~\int_0^t\left|K_H^{-1}\left(\int_0^\cdot Z(Y^\xi_r)\d r\right)(s)\right|^2\d s\geq n\right\},\qquad n\in\mathbb N.$$
Then
$$\tilde B_{n}(t):=B(t)-\int_0^{t\wedge \tau_n} K_H^{-1}\left(\int_0^\cdot Z(Y^\xi_r)\d r\right)(s)\d s,~t\geq 0$$
is a Brownian motion  under $R(T\wedge \tau_n)\P$. This implies
$$\tilde B^{H}_{n}(t):=B^H(t)-\int_0^{t\wedge \tau_n} Z(Y^\xi_r)\d r,~t\geq 0$$
is a fBM under $R^\xi({T\wedge \tau_n})\P$ and $Y^\xi$ satisfies 
$$\d Y^\xi(t)=b(Y^\xi(t))\d t+\sigma\d \tilde B^H_{n}(t)+\mathbf{1}_{[0\leq t\leq \tau_n]}\sigma Z(Y^\xi_t)\d t.$$
By \eqref{equ-hh-1}, following the proof of \eqref{ineq-Y-mon} and \eqref{add-ineq-beY}, we have
\begin{align*}
&|Y^\xi(t)|^2\leq |Y^\xi(0)|^2+2\int_0^t\left(\bar K_1+C_3\right)\|Y^\xi_s\|^2\d s\\
& +\int_0^t\left\{2(C_3+\bar K_1) \|\sigma\|\| \tilde B^H_{n,s}\|^2_\infty+\bar K_2C_1^2(1+|\sigma \tilde B^H_n(s)|^{q_0})^2+2C_3(1+\|\si\|\|\tilde B^H_{n,s}\|_{\infty}^{q_1})\right\}\d s
\end{align*}
and 
\begin{align*}
\|Y\|_{0,t,\beta}&\leq \left\{C_1\left(1+\|Y^\xi\|_{0,t,\infty}^{q_0}\right)+\|\sigma\| | Z(0)|+C_2\|\sigma\|\|Y^\xi\|_{-\tau,t,\infty}^{\alpha}\left(1+\|Y^\xi\|_{-\tau,t,\infty}^p\right)\right\}t^{1-\beta}\\
&\qquad  +\|\sigma\|\|\tilde  B_n^H\|_{0,t,\beta},~t>0,\beta\in (0,H).
\end{align*} 
Combining these with \eqref{ineq-ini} and that $\tilde B^{H}_{n}$ under $R^\xi(T\wedge\tau_n)\P$ has the same distribution as $B^H$ under $\P$, we have 
$$\sup_{n}\E R(T\wedge \tau_n)\left(\|Y^\xi\|_{-\tau,T,\infty}^q+\|Y^\xi\|_{-\tau,T,\theta }^q\right)<\infty,\qquad q>0, \, \theta<H.$$
Then by \eqref{ineq-KHZ}-\eqref{add-ineq-J3}, 
\begin{align*}
&\sup_{t\in [0,T], n}\E R^\xi(t\wedge\tau_n)\log R^\xi(t\wedge\tau_n)\\
&\qquad\leq C\sup_{t\in [0,T],n}\E R(t\wedge\tau_n)\left(1+\|Y^\xi\|_{-\tau,t,\infty}^{2(p+\alpha)}+\|Y^\xi\|_{-\tau,t,\theta}^{2\alpha}+\|Y^\xi\|_{-\tau,t,\theta}^{2\alpha}\|Y\|_{-\tau,t,\infty}^{2p}
\right)\\
&\qquad <\infty.
\end{align*}
Hence, it follows from the Fatou lemma and the martingale convergence theorem that $\{R^\xi(t)\}_{t\in [0,T]}$ is a uniformly integrable martingale and
$$\sup_{t\in [0,T]}\E R^\xi(t)\log R^\xi(t)<\infty.$$
It follows from Girsanov's theorem that under $R^\xi(T)\P$,  the process $\tilde B^H$ is a fBM.

(2) By \eqref{equ-hh-2}, we have 
\begin{align}\label{ineq-J1-J2}
\left|J_1(s)\right|^2+\left|J_2(s)\right|^2\leq C_4^2\left(\left(H-\frac 1 2\right)^{-2}+C_0^2\right)s^{1-2H}\left(\left(1+C_5\|Y^\xi\|_{-\tau,s,\infty}^{2p}\right)\wedge \|Y^\xi\|_{-\tau,s,\infty}^{2\alpha}\right).
\end{align} 
For $J_3$,  we can estimate
\begin{align*}
\left|J_3(s)\right|&\leq C_4 s^{H-\frac 1 2}\int_0^s\frac {\|Y^\xi_s-Y^\xi_r\|_\infty^\alpha\wedge\left(1+C_5(\|Y^\xi_s\|_\infty^p+\|Y_r\|_\infty^p)\right)} {(s-r)^{H+1/2} r^{H-1/2}}\d r \\
&\leq 2C_4s^{H-\frac 1 2}\int_0^s \frac {\left(1+C_5\|Y^\xi\|_{-\tau,s,\infty}^p\right)\wedge\left(\|Y^\xi\|^\alpha_{-\tau,s,\theta}(s-r)^{\theta\alpha}\right)} {(s-r)^{H+1/2}r^{H-1/2}}\d r \\
&\leq 2C_4s^{H-\frac 1 2}\left(1+C_5\|Y^\xi\|_{-\tau,s,\infty}^p\right)\int_0^s \frac { \|Y^\xi\|^\alpha_{-\tau,s,\theta}(s-r)^{\theta\alpha-H-1/2}r^{-H+1/2}} {1+C_5\|Y^\xi\|_{-\tau,s,\infty}^p+\|Y^\xi\|^\alpha_{-\tau,s,\theta}(s-r)^{\theta\alpha}}\d r \\
&\leq 2C_4s^{H-1/2}\left(1+C_5\|Y^\xi\|_{-\tau,s,\infty}^p\right)\\
&\qquad\times\left(\frac {2^{2H-1}s^{1-2H}} {3/2-H}+\frac {\theta\alpha 2^{H-1/2}s^{\frac 1 2-H}\|Y^\xi\|_{-\tau,s,\theta}^{\frac {H-1/2} {\theta}}\left(1+C_5\|Y^\xi\|_{-\tau,s,\infty}^{p}\right)^{-\frac {H-1/2} {\theta\alpha}}} {(\alpha\theta-H+1/2)(H-1/2)} \right),
\end{align*}
where we use \cite[Lemma 3.4]{FZ} in the last inequality. Thus
\begin{align}\label{ineq-J3-1}
\left|J_3(s)\right|^2\leq C_6 s^{1-2H}(1+C_5\|Y^\xi\|_{-\tau,s,\infty}^{p})^2+C_7\left(1+C_5\|Y^\xi\|_{-\tau,s,\infty}^p\right)^{\frac {2\theta\alpha-2H+1} {\theta\alpha}}\|Y^\xi\|_{-\tau,s,\theta}^{\frac {2H-1} {\theta}},
\end{align} 
where 
$$C_6=8C_4^2,\qquad C_7=\left(\frac {\theta\alpha 2^H} {(\alpha\theta-H+1/2)(H-1/2)}\right)^2.$$
Since $\|Y^\xi\|_{-\tau,s,\infty}\leq (s+\tau)^{\theta}\|Y^\xi\|_{-\tau,s,\theta}$, it follows from \eqref{ineq-J3-1} and \eqref{ineq-J1-J2} that 
\begin{align*}
&\E\exp\left\{C\int_0^T\left|K_H^{-1}\left(\int_0^\cdot Z(Y^\xi_r)\d r\right)(s)\right|^2\d s\right\}\\
&\leq \E \exp\left(C_T\left(1+\|Y^\xi\|_{-\tau,s,\theta}^{2p\vee \frac {2\theta\alpha p+(2H-1)(\alpha-p)} {\theta\alpha}}\right)\right)\\
&=\E \exp\left(C_T\left(1+\|Y^\xi\|_{-\tau,s,\theta}^{2p+ \frac { (2H-1)(\alpha-p)^+} {\theta\alpha}}\right)\right).
\end{align*}
For $p<1$, it is clear that
$$2p+ \frac { (2H-1)(\alpha-p)^+} {\theta\alpha}<2.$$
Then \eqref{Eexp} follows from  \eqref{add-ineq-beY} with $q_0=1,\beta=\theta$ and the Fernique-type lemma. 

(3) For $p=1$, substituting \eqref{ineq-J1-J2} and  \eqref{ineq-J3-1} into \eqref{ineq-KHZ}, we have 
\begin{align*}
\int_0^T\left|K_H^{-1}\left(\int_0^\cdot Z(Y^\xi_r)\d r\right)(s)\right|^2\d s&\leq C_8(T)+\frac {2C_4^2T^{2-2H}(1+(H-1/2)^2C_0^2)} {(1-H)\Gamma^2(3/2-H)}\|Y^\xi\|_{-\tau,T,\infty}^{2\alpha}\\
&\qquad+\frac {2C_6C_5^2T^{2-2T}(H-1/2)^2} {(1-H)\Gamma^2(3/2-H)}\|Y^\xi\|_{-\tau,T,\infty}^{2}\\
&\qquad+4C_7C_5^2T^{1+\theta}\Gamma^{-2}(3/2-H)\|Y^\xi\|_{-\tau,T,\theta}^{1+\frac {2H-1} {\theta}(1-\alpha^{-1})}.
\end{align*}
It follows from \eqref{ineq-Y-mon} and (A2) with $q_0=1$, we have
\begin{align*}
\|Y^\xi\|_{0,T,\infty}\leq \tilde C(T)+\left(C_1\Phi(\bar K_1,\bar K_2,T)+1\right)\|\sigma\|\|B^H\|_{0,T,\infty}.
\end{align*}
Therefore, for $T>0$ such that 
\begin{align*}
&\left(\frac {2C_4^2T^{2-2H}(1+(H-1/2)^2C_0^2)} {(1-H)\Gamma^2(3/2-H)}\mathbf{1}_{[\alpha=1]}+\frac {2C_6C_5^2T^{2-2T}(H-1/2)^2} {(1-H)\Gamma^2(3/2-H)}\right)\|\sigma\|^2\\
&\qquad <\left(C_1\Phi(\bar K_1,\bar K_2,T)+1\right)^{-2} \frac 1 {2T},
\end{align*}
it follows from Lemma \ref{F} that there is some $C>1$ such that
$$\E\exp\left\{C\int_0^T\left|K_H^{-1}\left(\int_0^\cdot Z(Y^\xi_r)\d r\right)(s)\right|^2\d s\right\} <\infty.$$
.

\end{proof}

\noindent{\bf{Proof of Theorem \ref{thm-ex-un}}}

We first show the existence of weak solution to \eqref{eq1.1}. It follows from (A1)-(A3) and Lemma \ref{lem-Eexph} that  $R^\xi(t)$
is an exponential martingale and $\tilde B^H(t)$ is a fBM under $R^\xi(T)\P$. Reformulating  the  reference SDE \eqref{eq1.6} as follows:
\begin{align}\label{eq1.8}
\d Y^\xi(t)=b(Y^\xi(t))\d t+\sigma Z(Y_t^\xi)\d t+\sigma\d\tilde{B}^H(t),
\end{align}
then under the complete filtration probability $(\Omega,\mathscr{F},(\mathscr{F}_t)_{t\in [0,T]},R^\xi(T)\P)$,  $\{Y^\xi(t)\}_{t\in [0,T]}$ is a  solution of \eqref{eq1.1}.

We shall  show the uniqueness of weak solutions to \eqref{eq1.1}.  For $i=1,2$, let  $(Y^{(i),\xi}(t))_{t\geq 0}$ be two weak solutions to \eqref{eq1.1} driven by  fBM $\left({B}_i^H(t)\right)_{t\geq 0}$ under the complete filtration probability
space $\left(\Omega^{(i)},\{\mathscr{F}_t^{(i)}\}_{t\ge0},\P^{(i)}\right)$ with $Y^{(i),\xi}_0=\xi$ satisfying \eqref{eq1.4}. Denote by $\P^{(i),\xi}$ the distribution of $Y^{(i),\xi}$. We intend to prove that  $\P^{(1),\xi}=\P^{(2),\xi}$. It follows from \eqref{eq1.1} that $Y^\xi(\cdot)\in C^\beta([0,T],\R^d)$ for any $\beta\in (0,H)$. Since $\xi\in C^\theta([-\tau,0],\R^d)$, $Y^\xi_\cdot\in C^{\theta\wedge\beta}([0,T],\mathscr{C})$. Let $\beta>\frac {2H-1} {2\alpha}$. Then $(\theta\wedge \beta)\alpha+\frac 1 2-H>0$. By \eqref{ineq-KHZ}-\eqref{add-ineq-J3}, we have 
$$\int_0^t\left|K_H^{-1}\left(\int_0^\cdot Z(Y^{(i),\xi}_r)\d r\right)(s)\right|^2\d s<\infty,~t\in [0,T],~\P\mbox{-a.s.}$$  

The rest of the proof can be complete by along the lines of the proof of \cite[Theorem 2.1]{W19}, we omit it here.

\section{Proof of Theorem \ref{thm1}}
Before giving the proof for Theorem \ref{thm1}, we  prepare two lemmas. The lemma below shows the estimates of uniform norm and H\"older norm of $(Y^\xi(t))_{t\in[0,T]}$ of the solution to \eqref{eq1.6}, respectively.
\begin{lem}\label{lem1.2}
Assume (H1). Then  for any $T>0$
\begin{align}\label{eq1.9}
\|Y^\xi\|_{-\tau,T,\8}&\le\|\xi\|_\infty+ |b(0)|\Phi(\bar K_1,\bar K_2,T)  +\left(L_1\Phi(\bar K_1,\bar K_2,T)+1\right)\|\sigma\|\|B^H\|_{\infty}. \\
\|Y^\xi\|_{-\tau,T,\beta\wedge\theta}&\le T^{1-\beta\wedge\theta}\left(|b(0)|+|b(0)|L_1\Phi(\bar K_1,\bar K_2,T)+|\xi(0)|\right)+\|\sigma\|\|B^H\|_{\beta\wedge\theta}\nonumber\\
&\qquad +L_1T^{1-\beta\wedge\theta}\left(L_1\Phi(\bar K_1,\bar K_2,T) +1\right)\|\sigma\|\|B^H\|_{\infty}+\|\xi\|_{-\tau,0,\beta\wedge\theta}.
\end{align}
\end{lem}
\begin{proof}
The first inequality follows from \eqref{ineq-Y-mon} and (H1) directly. Since $b$ is Lipschitz, 
$$|b(x)|\leq |b(0)|+L_1|x|.$$
Taking into account the following inequality
\begin{align*}
\|Y^\xi\|_{-\tau,T,\beta\wedge\theta}\le\|\xi\|_{-\tau,0,\beta\wedge\theta}+\|Y^\xi\|_{\beta\wedge\theta},
\end{align*}
the proof of the  second  inequality is  similar to the second part of the proof of Lemma \ref{lem-mom}.

\end{proof}

For the sake of simplicity, we denote $\Q^\xi=R^\xi(T)\P$,
\begin{align*}
h^\xi(t)=\hat{\sigma}^{-1}\{b(Y^\xi(t))-b(Y^\xi(t_\delta)\}-Z(\hat{Y}_t^\xi),~~t\ge0,
\end{align*} 
with 
$$\hat{Y}_t^\xi(u)=Y^\xi((t+u)\wedge t_\delta),~~u\in[-\tau,0].$$
Let
\begin{align}
 B_h^H(t)&=B^H(t)+\int_0^t h^\xi(s)\d s,\nonumber\\
R^{\xi,\delta}(t)&=\exp\Big\{-\int_0^t\Big\langle K_H^{-1}\big(\int_0^\cdot h^\xi(s)\d s\big)(r),\d B(r)\Big\rangle\nonumber\\
&\qquad\qquad-\frac{1}{2}\int_0^t\Big|K_H^{-1}\big(\int_0^\cdot h^\xi(s)\d s\big)(r)\Big|^2\d r\Big\}, t\in [0,T],\label{eq2.4}
\end{align}
and let $\d\Q^{\xi,\delta}=R^{\xi,\delta}(T)\d \P$. Then it follows from Lemma \ref{lem1.35} below and the Girsanov theorem that $\Q^{\xi,\delta}$ is a probability and $(B_h^H(t))_{t\in[0,T]}$ is a  fBM under $\Q^{\xi,\delta}$.
Since $\sigma\sigma^{-1}=\pi_*$, we can  rewrite  the reference SDE \eqref{eq1.6}  into the following form
\begin{align}\label{eq1.11}
\d Y^\xi(t)=\{(I_{d\times d}-\pi_*)b(Y^\xi(t))+\pi_*b(Y^\xi(t_\delta))+\sigma Z(\hat{Y}_{t}^\xi)\}\d t+\sigma\d B_h^H(t),
\end{align}
which implies  that $(Y^\xi(t),B_h^H(t))_{t\in[0,T]}$ is a weak solution of \eqref{eq1.5}.
Obviously, \eqref{eq1.5} has a unique pathwise solution,  so the weak uniqueness follows.  Then
$$|\E f(X(t))-\E f(X^{(\delta)}(t))|=|\E_{\Q^\xi}f(Y^\xi(t))-\E_{\Q^{\xi,\delta}}f(Y^\xi(t))|=|\E(R^\xi(t)-R^{\xi,\delta}(t))f(Y^\xi(t))|.$$
Hence, in the following discussion, we shall prove that $\{R^{\xi,\delta}(t)\}_{t\in [0,T]}$ is a exponential martingale and give estimates of $R^\xi(t)-R^{\xi,\delta}(t)$.

\begin{lem}\label{lem1.35}
Under the assumptions of Theorem \ref{thm1}, we have 
$$\E\exp\left\{C\int_0^t|K_H^{-1}\big(\int_0^\cdot h^\xi(s)\d s\big)(r)|^2\d r\right\}< \infty,$$
for some $C>1$.
\end{lem}

\begin{proof}
The definition of inverse operator $K_H^{-1}$ yields that
\begin{align}\label{eqh1}
K_H^{-1}&\Big(\int_0^\cdot h^\xi(s)\d s\Big)(r)=r^{H-\frac{1}{2}}D_{0+}^{H-\frac{1}{2}}[\cdot^{\frac{1}{2}-H}h^\xi(\cdot)](r)\nonumber\\
&=r^{H-\frac{1}{2}}\frac{1}{\Gamma(\frac{3}{2}-H)}\Big(
\frac{r^{\frac{1}{2}-H}h_1^\xi(r)}{r^{H-\frac{1}{2}}}+(H-\frac{1}{2})\int_0^r\frac{r^{\frac{1}{2}-H}h^\xi(r)-s^{\frac{1}{2}-H}h^\xi(s)}
{(r-s)^{H+\frac{1}{2}}}\d s\Big)\nonumber\\
&=\frac{r^{\frac{1}{2}-H}h_1^\xi(r)}{\Gamma(\frac{3}{2}-H)}
+\frac{r^{H-\frac{1}{2}}}{\Gamma(\frac{3}{2}-H)}(H-\frac{1}{2})
\int_0^r\frac{r^{\frac{1}{2}-H}h^\xi(r)-s^{\frac{1}{2}-H}h^\xi(s)}
{(r-s)^{H+\frac{1}{2}}}\d s\nonumber\\
&=\frac{r^{\frac{1}{2}-H}h^\xi(r)}{\Gamma(\frac{3}{2}-H)}
+\frac{r^{H-\frac{1}{2}}}{\Gamma(\frac{3}{2}-H)}(H-\frac{1}{2})
\int_0^r\frac{(r^{\frac{1}{2}-H}-s^{\frac{1}{2}-H})h^\xi(r)}
{(r-s)^{H+\frac{1}{2}}}\d s\nonumber\\
&\qquad +\frac{r^{H-\frac{1}{2}}}{\Gamma(\frac{3}{2}-H)}(H-\frac{1}{2})
\int_0^r\frac{s^{\frac{1}{2}-H}(h^\xi(r)-h^\xi(s))}
{(r-s)^{H+\frac{1}{2}}}\d s\nonumber\\
&=[1+C_0(H-\frac{1}{2})]\frac{r^{\frac{1}{2}-H}h^\xi(r)}{\Gamma(\frac{3}{2}-H)}
+\frac{r^{H-\frac{1}{2}}}{\Gamma(\frac{3}{2}-H)}(H-\frac{1}{2})
\int_0^r\frac{s^{\frac{1}{2}-H}(h^\xi(r)-h^\xi(s))}
{(r-s)^{H+\frac{1}{2}}}\d s\nonumber\\
&=: \hat J_1(r)+\hat J_2(r).
\end{align}
For $\hat J_1$, it follows from  (H1) and (H2) that 
\begin{align*}
|h^\xi(r)|
&\le\|\hat{\sigma}^{-1}\||b(Y^\xi(r))-b(Y^\xi(r_\delta))|+|Z(\hat{Y}_r^\xi)|\\
&\le\|\hat{\sigma}^{-1}\|L_1\|Y^\xi\|_{0,r,\beta}\delta^\beta+|Z(0)|+L_2\|Y^\xi\|_{-\tau,r,\8}^\alpha.
\end{align*}
For $\hat J_2$, note that
\begin{align*}
&\int_0^r\frac{s^{\frac{1}{2}-H}|h^\xi(r)-h^\xi(s)|}
{(r-s)^{H+\frac{1}{2}}}\d s\\
=&\int_0^r\frac{s^{\frac{1}{2}-H}
|\hat{\sigma}^{-1}(b(Y^\xi(r))-b(Y^\xi(r_\delta)))-Z(\hat{Y}_r^\xi)-
\hat{\sigma}^{-1}(b(Y^\xi(s))-b(Y^\xi(s_\delta)))+Z(\hat{Y}_s^\xi))|}
{(r-s)^{H+\frac{1}{2}}}\d s\\
\le&\int_0^r\frac{\|\hat{\sigma}^{-1}\|s^{\frac{1}{2}-H}
|b(Y^\xi(r))-b(Y^\xi(r_\delta))-
(b(Y^\xi(s))-b(Y^\xi(s_\delta)))|}
{(r-s)^{H+\frac{1}{2}}}\d s\\
&~~~~+\int_0^r\frac{s^{\frac{1}{2}-H}
|Z(\hat{Y}_r^\xi)-Z(\hat{Y}_s^\xi))|}
{(r-s)^{H+\frac{1}{2}}}\d s\\
&=: I_1(r)+I_2(r).
\end{align*}

Next, we shall give the estimate of $I_i(r), i=1,2$. For $I_1(r)$, it follows from (H1) that
\begin{align*}
&|b(Y^\xi(r))-b(Y^\xi(r_\delta))-
(b(Y^\xi(s))-b(Y^\xi(s_\delta)))|\\
&\qquad\le 2L_1\|Y^\xi\|_{0,r,\beta}\Big[\delta^\beta\wedge\frac{(r-s)^\beta+(r_\delta-s_\delta)^\beta}{2}\Big]\\
&\qquad=L_1\|Y^\xi\|_{0,r,\beta}\left\{
             \begin{array}{lr}
            {(r-s)^\beta}, & r_\delta<s<r,   \\
             {(r-s)^\beta+(r_\delta-s_\delta)^\beta}, & r-\delta<s<r_\delta,\\
             2\delta^\beta, &  0<s<r-\delta.\\
             \end{array}
\right.
\end{align*}
Since
\begin{align*}
|r_\delta-s_\delta|=|[\frac{r}{\delta}]\delta-[\frac{s}{\delta}]\delta|\le|[\frac{r}{\delta}]\delta-[\frac{r-\delta}{\delta}]\delta|
\le\delta,    ~~~~r-\delta<s<r_\delta,
\end{align*}
and for $r\geq \delta$
\begin{align*}
&\int_{r_\delta}^r\frac{s^{1/2-H}}{(r-s)^{H+1/2-\beta}}\d s
\le\frac{2\delta^{1/2+\beta-H}}{1+2\beta-2H} {r_\delta^{1/2-H}},\\
&\int_0^{r_\delta}\frac{2\delta^\beta}{(r-s)^{H+1/2-\beta}s^{H-1/2}}\d s\\
&=\int_{r_\delta/2}^{r_\delta} \frac {2\delta^\beta} {(r-s)^{1/2+H}s^{H-1/2}}\d s+\int_0^{r_\delta/2}\frac{2\delta^\beta} {(r-s)^{1/2+H}s^{H-1/2}}\d s\\
&\le\frac {2\delta^\beta(r-r_\delta)^{\frac 1 2-H}} {(r_\delta/2)^{H-1/2}(H-1/2)}+\frac {2\delta^\beta\left(\frac {r_\delta} 2\right)^{\frac 3 2-H}} {\left(r-r_\delta/2\right)^{H+1/2}\left(\frac 3 2-H\right)},
\end{align*}
we have
\begin{align*}
|I_1(r)|&\le2L_1\|\hat{\sigma}^{-1}\|\|Y^\xi\|_{0,r,\beta} \left\{\left[
\frac{\delta^{\beta+1/2-H}}{(2\beta+1-2H)r_\delta^{H-1/2}}+
\frac {\delta^\beta(r-r_\delta)^{\frac 1 2-H}} {(r_\delta/2)^{H-1/2}(H-1/2)}\right.\right.\\
& \left.\left. +\frac {\delta^\beta\left(\frac {r_\delta} 2\right)^{\frac 3 2-H}} {\left(r-r_\delta/2\right)^{H+1/2}\left(\frac 3 2-H\right)}\right]\1_{[r\geq \delta]}+\frac 1 2\mathcal{B}(\frac 3 2-H, \beta+\frac 1 2-H) r^{\beta+1-2H}\1_{[0\leq r<\delta]} \right\}.
\end{align*}
We now calculate $I_2(r).$  One can see that
\begin{align*}
&\|\hat{Y}_r^\xi-\hat{Y}_s^\xi\|^\alpha\\
&=\sup_{-\tau\le u\le0}\frac{|Y^\xi((r+u)\wedge r_\delta)-Y^\xi((s+u)\wedge s_\delta)|^\alpha}
{|(r+u)\wedge r_\delta-(s+u)\wedge s_\delta|^{\alpha(\beta\wedge\theta)}}
|(r+u)\wedge r_\delta-(s+u)\wedge s_\delta|^{\alpha(\beta\wedge\theta)}\\
&\le\|Y^\xi\|_{-\tau,r,\beta\wedge\theta}^\alpha
\sup_{-\tau\le u\le0}|(r+u)\wedge r_\delta-(s+u)\wedge s_\delta|^{\alpha(\beta\wedge\theta)}.
\end{align*}
Since for $s+u>s_\delta$ and $r+u<r_\delta$, we have
$$(s+u)\wedge s_\delta=s_\delta;\qquad  (r+u)\wedge r_\delta=r-u;\qquad s_\delta-s<u<r_\delta-r.$$
Then 
$$\sup_{s_\delta-s<u<r_\delta-r}\left|(r+u)\wedge r_\delta-(s+u)\wedge s_\delta\right|=\sup_{s_\delta-s<u<r_\delta-r}\left|r+u - s_\delta\right|=|r_\delta-s_\delta|.$$
Similarly, for $s+u<s_\delta$ and $r+u>r_\delta$, we have
$$\sup_{r_\delta-r<u<s_\delta-s}\left|(r+u)\wedge r_\delta-(s+u)\wedge s_\delta\right|=|r_\delta-s_\delta|.$$
Then it is easy to see that 
$$\sup_{u\in [-\tau,0]}\left|(r+u)\wedge r_\delta-(s+u)\wedge s_\delta\right|=(t-s)\vee(t_\delta-s_\delta).$$
Consequently
$$\|\hat{Y}_r^\xi-\hat{Y}_s^\xi\|^\alpha\leq \|Y^\xi\|_{-\tau,r, \beta\wedge\theta }^\alpha((t-s)\vee(t_\delta-s_\delta))^{\alpha(\beta\wedge\theta)},$$
and 
\begin{align*}
I_2(r)&=\int_0^r\frac{s^{\frac{1}{2}-H}
|Z(\hat{Y}_r^\xi)-Z(\hat{Y}_s^\xi))|}
{(r-s)^{H+\frac{1}{2}}}\d s\\
&\le\int_0^r\frac{L_2\|\hat{Y}_r^\xi-\hat{Y}_s^\xi\|^\alpha}{(r-s)^{1/2+H}s^{H-1/2}}\d s\\
&\le L_2\|Y\|_{-\tau,r, \beta\wedge\theta }^\alpha\int_0^r\ff {(t-s)^{\alpha(\beta\wedge\theta)}\vee(t_\delta-s_\delta)^{\alpha(\beta\wedge\theta)}} {(r-s)^{1/2+H}s^{H-1/2}}\d s.
\end{align*}
Since $r_\delta-s_\delta=0$ for $s\in [r_\delta,r]$, 
\begin{align*}
\int_0^r\ff {(r-s)^{\alpha(\beta\wedge\theta)}\vee(r_\delta-s_\delta)^{\alpha(\beta\wedge\theta)}} {(r-s)^{1/2+H}s^{H-1/2}}\d s&= \int_0^r\ff {(r-s)^{\alpha(\beta\wedge\theta)}} {(r-s)^{1/2+H}s^{H-1/2}}\mathbf{1}_{[r-s\geq r_\delta-s_\delta]}\d s\\
&+\int_0^{r_\delta}\ff {(r_\delta-s_\delta)^{\alpha(\beta\wedge\theta)}} {(r-s)^{1/2+H}s^{H-1/2}}\mathbf{1}_{[r-s< r_\delta-s_\delta]}\d s.
\end{align*}
For $r- r_\delta+s_\delta<s$, it is clear that $r_\delta-s_\delta-(r-s)\leq \delta$, so
$$(r_\delta-s_\delta)^{\alpha(\beta\wedge\theta)}=(r_\delta-s_\delta-r+s+(r-s))^{\alpha(\beta\wedge\theta)}\leq (r-s)^{\alpha(\beta\wedge\theta)}+\delta^{\alpha(\beta\wedge\theta)},$$
which implies that 
\begin{align*}
\int_0^{r_\delta}\ff {(r_\delta-s_\delta)^{\alpha(\beta\wedge\theta)}} {(r-s)^{1/2+H}s^{H-1/2}}\mathbf{1}_{[r-s< r_\delta-s_\delta]}\d s\leq \int_0^{r_\delta}\ff {(t-s)^{\alpha(\beta\wedge\theta)}+\delta^{\alpha(\beta\wedge\theta)}} {(r-s)^{1/2+H}s^{H-1/2}}\mathbf{1}_{[r-s< r_\delta-s_\delta]}\d s.
\end{align*} 
Moreover, we have 
\begin{align}\label{ineq-de-sr}
&\int_{0}^T\left(\int_0^{r_\delta}\frac {\delta^{\alpha(\beta\wedge\theta)}\mathbf{1}_{[t-s<t_\delta-s_\delta]}} {(r-s)^{1/2+H}s^{H-1/2}}\d s\right)^2\d r\leq \int_{\delta}^T\left(\int_0^{r_\delta}\frac {\delta^{\alpha(\beta\wedge\theta)}} {(r-s)^{1/2+H}s^{H-1/2}}\d s\right)^2\d r\nonumber\\
&\qquad \leq \sum_{k=1}^{N-1}\int_{k\delta}^{(k+1)\delta}\left(\frac {\delta^{\alpha(\beta\wedge\theta)}\left(\frac {r_\delta} 2\right)^{\frac 3 2-H}} {\left(r-r_\delta/2\right)^{H+1/2}\left(\frac 3 2-H\right)}+\frac {\delta^{\alpha(\beta\wedge\theta)}(r-r_\delta)^{\frac 1 2-H}} {(r_\delta/2)^{H-1/2}(H-1/2)}\right)^2\d r\nonumber\\
&\qquad \leq 2\delta^{2\alpha(\beta\wedge\theta)}\sum_{k=1}^{N-1}\left(\frac {16^{H}(k\delta)^{2-4H}\delta} {(3-2H)^2}+\frac {2^{2H-1}\delta^{2-2H}}{(H-1/2)^2(k\delta)^{2H-1}}\right)\nonumber\\
&\qquad \leq 2\delta^{2\alpha(\beta\wedge\theta)+1-2H}\left(\frac {16^H} {(3-2H)^2}+\frac {2^{2H+1}} {(2H-1)^2}\right)\sum_{k=1}^{N-1}(k\delta)^{1-2H}\delta \nonumber\\
&\qquad \leq \frac{\delta^{2\alpha(\beta\wedge\theta)+1-2H}}{1-H}\left(\frac {16^H} {(3-2H)^2}+\frac {2^{2H+1}} {(2H-1)^2}\right)T^{2-2H}
\end{align}
and 
\begin{align}\label{add-ineq-sr1}
&\int_0^T\left(\int_0^r\ff {(r-s)^{\alpha(\beta\wedge\theta)}} {(r-s)^{1/2+H}s^{H-1/2}}\d s\right)^2\d r
=\frac {T^{2\alpha(\beta\wedge\theta)+3-4H}\mathcal{B}^2(\frac 3 2-H,\alpha(\beta\wedge\theta)+\frac 1 2-H)} {2\alpha(\beta\wedge\theta)+3-4H}.
\end{align}
Substituting  $\hat J_1$, $I_1(r)$ and $I_2(r)$ into \eqref{eqh1}, and taking into account \eqref{ineq-de-sr} and \eqref{add-ineq-sr1}, we arrive at 
\begin{align}\label{add-ineq-hK}
&\int_0^T\Big|K_H^{-1}\Big(\int_0^\cdot h^\xi(s)\d s\Big)(r)\Big|^2\d r\nonumber\\
&\qquad\le\frac{8L_1^2\|\hat{\sigma}^{-1}\|^2\|Y^\xi\|_{\beta}^2}{\Gamma^2({\frac{3}{2}-H)}(1-H)}\Big(\left[1+C_0(H-\frac{1}{2})\right]^2T^{2-2H}\delta^{2\beta}\nonumber\\
&\qquad\qquad+8\delta^{2\beta+1-2H}T\Big[
\frac{1}{(1+2\beta-2H)^2}+\frac{16^H}{(3-2H)^2}+\frac{2^{2H+1}}{(2H-1)^2}
\Big]+\frac{4(1-H)\delta^{2\beta+4-4H}}{(\beta+2-2H)^2}\Big)\nonumber\\
&\qquad\qquad+\frac{2[1+C_0(H-\frac{1}{2})]^2|Z(0)|^2T^{2(1-H)}}{\Gamma^2(\frac{3}{2}-H)(1-H)}
+\frac{2L_2^2[1+C_0(H-\frac{1}{2})]^2T^{2-2H}\|Y^\xi\|_{-\tau,T,\8}^{2\alpha}}{\Gamma^2(\frac{3}{2}-H)(1-H)}\nonumber\\
&\qquad\qquad+\frac{12L_2^2\|Y^\xi\|_{-\tau,T, \beta\wedge\theta }^{2\alpha}}{\Gamma^2(\frac{3}{2}-H)}
\Big[\frac{\mathcal{B}^2(\frac{3}{2}-H,\alpha(\beta\wedge\theta)+1/2-H)T^{2\alpha(\beta\wedge\theta)+3-4H}}{2\alpha(\beta\wedge\theta)+3-4H}\nonumber\\
&\qquad\qquad+\frac{\delta^{2\alpha(\beta\wedge\theta)+1-2H}T^{2-2H}}{1-H}
\Big(\frac{16^H}{(3-2H)^2}+\frac{2^{2H+1}}{(2H-1)^2}
\Big)\Big].
\end{align}
Therefore, it follows from Lemma \ref{lem1.2} and \eqref{eq2.5} that there exists $C>1$ such that
\begin{align*}
&\E\exp\Big\{C\int_0^T\Big|K_H^{-1}\Big(\int_0^\cdot h^\xi(s)\d s\Big)(r)\Big|^2\d r\Big\}<\infty.
\end{align*}

\end{proof}

We are now in the position to complete the 

\noindent{\bf Proof of Theorem \ref{thm1}}.  Let
\begin{align*}
&M_1(t)=\int_0^t\Big\langle K_H^{-1}\Big(\int_0^\cdot Z(Y_s^\xi)\d s\Big)(r),\d B(r)\Big\rangle,\\
&M_2(t)=\int_0^t\Big\langle K_H^{-1}\Big(\int_0^\cdot h^\xi(s)\d s\Big)(r),d B(r)\Big\rangle,  ~~~t\ge0.
\end{align*}
For $f\in\mathscr{B}_b(\R^d)$, following from the weak uniqueness of solution to \eqref{eq1.1}, the  H\"older's inequality and the following inequality  
$$|\e^x-\e^y|\le(\e^x\vee \e^y)|x-y|,$$
we have
\begin{align}\label{eq1.12}
|\E f(X(t))-\E f(X^{(\delta)}(t))|&=|\E_{\Q^\xi}f(Y^\xi(t))-\E_{\Q^{\xi,\delta}}f(Y^\xi(t))|\nonumber\\
&=|\E(R^\xi(t)-R^{\xi,\delta}(t))f(Y^\xi(t))|\nonumber\\
&\le\|f\|_\8\E|R^\xi(t)-R^{\xi,\delta}(t)|\nonumber\\
&\le\|f\|_\8\E\left(R^\xi(t)\vee R^{\xi,\delta}(t)\right)
\left|\log R^\xi(t)-\log R^{\xi,\delta}(t)\right|\nonumber
\\
&\leq \|f\|_\8\Theta_1(t)(\Theta_2(t)+\Theta_3(t)), t\in[0,T],
\end{align}
where 
\begin{align*}
\Theta_1(t)&=\left(\E (R^\xi(t))^q\right)^{\frac 1 q}+\left(\E (R^{\xi,\delta}(t))^q\right)^{\frac 1 q},\\
\Theta_2(t)&=\left(\E\left|\int_0^t\<K_H^{-1}\left(\int_0^\cdot(Z(Y^\xi_s)+h^\xi(s))\d s\right)(r),\d B(r) \>\right|^{\frac {q} {q-1}}\right)^{\frac {q-1} q},\\
\Theta_3(t)&= \frac 1 2\left(\E\left|\int_0^t\left(\left|K_H^{-1}\left(\int_0^\cdot Z(Y^\xi_s)\d s\right)(r)\right|^2-\left|K_H^{-1}\left(\int_0^\cdot h^\xi(s)\d s\right)(r)\right|^2\right)\d r\right|^{\frac q {q-1}}\right)^{\frac {q-1} q},~q>1.
\end{align*}

It follows from \eqref{add-fZ} and Lemma \ref{lem-Eexph} with $C_5=0$ and $C_4=L_2$ that there is some $C>1$ such that $\E\exp\{C\<M_1\>(T)\}<\infty$. Thus,  for $2q^2-q\leq C$, we have 
\begin{align*}
\E(R^\xi(t))^q&=\E\exp\Big(qM_1(t)-q^2\langle M_1\rangle(t)+(q^2-q/2)\langle M_1\rangle(t)\Big)\\
&\le(\E\exp(2qM_1(t)-2q^2\langle M_1\rangle(t)))^{1/2}
\big(\E\exp((2q^2-q)\langle M_1\rangle(t))\big)^{1/2}\\
&\le \Big(\E\exp\Big((2q^2-q)\int_0^t\Big|K_H^{-1}\Big(\int_0^\cdot  Z(Y_s^\xi)\d s\Big)(r)\Big|^2\d r\Big)\Big)^{1/2}\\
&<\infty.
\end{align*}
Similarly, following from Lemma \ref{lem1.35}, there is  $q>1$ such that
$$\sup_{t\in [0,T]}\left(\E(R^{\xi,\delta}(t))^q\right)^{\frac 1 q}<\infty.$$
Hence,  there is $q>1$ and   some constant $C_{T}$ such that
\begin{align}\label{eq2.6}
\Theta_1(t)\le C_{T}.
\end{align}
In the following proof, we fix some $q>1$ such that \eqref{eq2.6} holds.

It is easy to see that
\begin{align}\label{add-ineq-Th2}
&K_H^{-1}\Big(\int_0^\cdot( Z(Y_s^\xi)+h^\xi(s))\d s\Big)(r)=r^{H-\frac{1}{2}}D_{0+}^{H-\frac{1}{2}}[\cdot^{\frac{1}{2}-H}( Z(Y_\cdot^\xi)+h^\xi(\cdot))](r)\nonumber\\
&\qquad=\frac{r^{H-\frac{1}{2}}}{\Gamma(\frac{3}{2}-H)}
\Big(\frac{r^{1/2-H}( Z(Y_r^\xi)+h^\xi(r))}{r^{H-1/2}}\nonumber\\
&\qquad\qquad+(H-\frac{1}{2})\int_0^r\frac{r^{1/2-H}( Z(Y_r^\xi)+h^\xi(r))-s^{1/2-H}( Z(Y_s^\xi)+h^\xi(s))}{(r-s)^{H+1/2}}\d s\Big)\nonumber\\
&\qquad\le[1+C_0(H-\frac{1}{2})]\frac{r^{\frac{1}{2}-H}( Z(Y_r^\xi)+h^\xi(r))}{\Gamma(\frac{3}{2}-H)}\nonumber\\
&\qquad\qquad+\frac{r^{H-\frac{1}{2}}}{\Gamma(\frac{3}{2}-H)}(H-\frac{1}{2})
\int_0^r\frac{s^{\frac{1}{2}-H}( Z(Y_r^\xi)+h^\xi(r)- Z(Y_s^\xi)-h^\xi(s))}
{(r-s)^{H+\frac{1}{2}}}\d s\nonumber\\
&\qquad=: I_3(r)+I_4(r).
\end{align}
Next, we give the estimates for  $I_i(r), i=3,4$, respectively.
  (H1) and (H2) yields that
\begin{align}\label{add-ineq-I3}
|I_3(r)|&\le[1+C_0(H-\frac{1}{2})]\frac{r^{\frac{1}{2}-H}}{\Gamma(\frac{3}{2}-H)}
 (\|\hat{\sigma}^{-1}\||b(Y^\xi(r))-b(Y^\xi(r_\delta))|+\|Z(Y_r^\xi)-Z(\hat{Y}_r^\xi)\|)\nonumber\\
&\le[1+C_0(H-\frac{1}{2})]\frac{r^{\frac{1}{2}-H}}{\Gamma(\frac{3}{2}-H)}
 (L_1\|\hat{\sigma}^{-1}\|\|Y^\xi\|_{0,r,\beta}\delta^\beta+L_2\|Y_r^\xi-\hat{Y}_r^\xi\|^\alpha)\nonumber\\
&\le \frac{\left[1+C_0(H-\frac{1}{2})\right]r^{\frac{1}{2}-H}}{\Gamma(\frac{3}{2}-H)}
\left(L_1\|\hat{\sigma}^{-1}\|\|Y^\xi\|_{0,r,\beta}\delta^\beta
+ L_2\|Y^\xi\|_{0,r,\beta}^\alpha\delta^{\alpha\beta}\right).
\end{align}
Moreover, we have
\begin{align}\label{add-ineq-I4}
|I_4(r)|&\le\frac{r^{H-\frac{1}{2}}}{\Gamma(\frac{3}{2}-H)}(H-\frac{1}{2})
\int_0^r\frac{s^{\frac{1}{2}-H}(Z(Y_r^\xi)+h^\xi(r)- Z(Y_s^\xi)-h^\xi(s))}
{(r-s)^{H+\frac{1}{2}}}\d s\nonumber\\
&\le\frac{r^{H-\frac{1}{2}}}{\Gamma(\frac{3}{2}-H)}(H-\frac{1}{2})\|\hat{\sigma}^{-1}\| \int_0^r\frac{s^{1/2-H}|b(Y^\xi(r))-b(Y^\xi(r_\delta))-(b(Y^\xi(s))-b(Y^\xi(s_\delta)))|}{(r-s)^{H+1/2}}\d s\nonumber\\
&~~+\frac{r^{H-\frac{1}{2}}}{\Gamma(\frac{3}{2}-H)}(H-\frac{1}{2})\int_0^r\frac{s^{1/2-H}\|Z(Y_r^\xi)-Z(\hat{Y}_r^\xi)-(Z(Y_s^\xi)-Z(\hat{Y}_s^\xi))\|}{(r-s)^{H+1/2}}\d s\nonumber\\
&=I_{41}(r)+I_{42}(r).
\end{align}
In the same way to estimate $I_1$ in the proof of Lemma \ref{lem1.35}, we have
\begin{align}\label{add-ineq-I41}
I_{41}(r)&\le2\frac{r^{H-\frac{1}{2}}}{\Gamma(\frac{3}{2}-H)}(H-\frac{1}{2})L_1\|\hat{\sigma}^{-1}\|\|Y^\xi\|_{0,r,\beta}\left\{\Big[
\frac{\delta^{\beta+1/2-H}}{(2\beta+1-2H)r_\delta^{H-1/2}}+
\frac {\delta^\beta(r-r_\delta)^{\frac 1 2-H}} {(r_\delta/2)^{H-1/2}(H-1/2)}\right.\nonumber\\
&~~\left. +\frac {\delta^\beta\left(\frac {r_\delta} 2\right)^{\frac 3 2-H}} {\left(r-r_\delta/2\right)^{H+1/2}\left(\frac 3 2-H\right)}
\Big]\1_{[ r\ge \delta]}+\mathcal{B}(\beta+\frac 1 2-H,\frac 3 2-H) r^{\beta+1-2H}\1_{[0\leq r<\delta]} \right\}.
\end{align}
On the other hand,  it follows from (H2) that
\begin{align*}
\|Z(Y_r^\xi)-Z(\hat{Y}_r^\xi)-(Z(Y_s^\xi)-Z(\hat{Y}_s^\xi))\|&\le L_2\|Y_r^\xi-\hat{Y}_r^\xi\|^\alpha+L_2\|Y_s^\xi-\hat{Y}_s^\xi\|^\alpha\\
&\le 2L_2\|Y\|_{-\tau,r, \beta\wedge\theta }^\alpha\delta^{\alpha(\beta\wedge\theta)},
\end{align*}
and
\begin{align*}
&\|Z(Y_r^\xi)-Z(\hat{Y}_r^\xi)-(Z(Y_s^\xi)-Z(\hat{Y}_s^\xi))\|\\
&\qquad\le L_2\|Y_r^\xi-Y_s^\xi\|^\alpha+L_2\|\hat{Y}_r^\xi-\hat{Y}_s^\xi\|^\alpha\\
&\qquad\le L_2\|Y\|_{-\tau,r, \beta\wedge\theta }^\alpha|r-s|^{\alpha(\beta\wedge\theta)}+L_2\|Y\|_{-\tau,r, \beta\wedge\theta }^\alpha\Big(|r-s|^{\alpha\beta}\vee |r_\delta-s_\delta|^{\alpha(\beta\wedge\theta)}\Big)\\
&\qquad=L_2\|Y\|_{-\tau,r, \beta\wedge\theta }^\alpha\Big(|r-s|^{\alpha(\beta\wedge\theta)}+|r-s|^{\alpha(\beta\wedge\theta)}\vee |r_\delta-s_\delta|^{\alpha(\beta\wedge\theta)}\Big).
\end{align*}
Combining these two upper bounds together, we have
\begin{align*}
&\|Z(Y_r^\xi)-Z(\hat{Y}_r^\xi)-(Z(Y_s^\xi)-Z(\hat{Y}_s^\xi))\|\\
&\qquad\le 2L_2\|Y^\xi\|_{-\tau,r,\beta}^\alpha
\Big(\delta^{\alpha(\beta\wedge\theta)}\wedge\frac{|r-s|^{\alpha(\beta\wedge\theta)}+|r-s|^{\alpha(\beta\wedge\theta)}\vee |r_\delta-s_\delta|^{\alpha\beta}}{2}\Big).
\end{align*}
Since 
$$\delta^{\alpha(\beta\wedge\theta)}\wedge\frac{|r-s|^{\alpha(\beta\wedge\theta)}+|r-s|^{\alpha\beta}\vee |r_\delta-s_\delta|^{\alpha(\beta\wedge\theta)}}{2}=\delta^{\alpha(\beta\wedge\theta)},s\in [0,r_\delta],$$
we get
\begin{align*}
&\int_0^{r_\delta}\frac {\delta^{\alpha(\beta\wedge\theta)}\wedge\frac{|r-s|^{\alpha(\beta\wedge\theta)}+|r-s|^{\alpha(\beta\wedge\theta)}\vee |r_\delta-s_\delta|^{\alpha(\beta\wedge\theta)}}{2}} {(r-s)^{H+1/2}s^{H-1/2}}\d s\\
&\qquad=\int_0^{r_\delta}\frac {\delta^{\alpha(\beta\wedge\theta)}} {(r-s)^{H+1/2}s^{H-1/2}}\d s\\
&\qquad\le\frac {\delta^{(\beta\wedge\theta)}(r-r_\delta)^{\frac 1 2-H}} {(r_\delta/2)^{H-1/2}(H-1/2)}+\frac {\delta^{(\beta\wedge\theta)}\left(\frac {r_\delta} 2\right)^{\frac 3 2-H}} {\left(r-r_\delta/2\right)^{H+1/2}\left(\frac 3 2-H\right)}
\end{align*}
and
\begin{align*}
\int_{r_\delta}^r\frac {\delta^{\alpha{(\beta\wedge\theta)}}\wedge\frac{|r-s|^{\alpha{(\beta\wedge\theta)}}+|r-s|^{\alpha{(\beta\wedge\theta)}}\vee |r_\delta-s_\delta|^{\alpha{(\beta\wedge\theta)}}}{2}} {(r-s)^{H+1/2}s^{H-1/2}}\d s & =\int_{r_\delta}^r\frac {(r-s)^{\alpha{(\beta\wedge\theta)}}} {(r-s)^{H+1/2}s^{H-1/2}}\d s\\
&\leq \frac {r_\delta^{\frac 1 2-H}(r-r_\delta)^{\alpha{(\beta\wedge\theta)}+\frac 1 2-H}} {\alpha{(\beta\wedge\theta)}+\frac 1 2-H}\\
&\leq \frac {2r_\delta^{\frac 1 2-H}\delta^{\alpha{(\beta\wedge\theta)}+\frac 1 2-H}} {2\alpha{(\beta\wedge\theta)}+1-2H}.
\end{align*}
Thus, 
\begin{align}\label{add-ineq-I42}
&|I_{42}(r)|\le\frac{2 L_2(H-\frac{1}{2}) r^{H-\frac{1}{2}}}{\Gamma(\frac{3}{2}-H)}\|Y^\xi\|_{-\tau,r,{ \beta\wedge\theta }}^\alpha\left\{\Big(
\frac {\delta^{(\beta\wedge\theta)}(r-r_\delta)^{\frac 1 2-H}} {(r_\delta/2)^{H-1/2}(H-1/2)}+\frac {\delta^{(\beta\wedge\theta)}\left(\frac {r_\delta} 2\right)^{\frac 3 2-H}} {\left(r-r_\delta/2\right)^{H+1/2}\left(\frac 3 2-H\right)}\right.\nonumber\\
&\left. +\frac {2r_\delta^{\frac 1 2-H}\delta^{\alpha{(\beta\wedge\theta)}+\frac 1 2-H}} {2\alpha{(\beta\wedge\theta)}+1-2H}\Big)\1{[r\geq \delta]}+\mathcal{B}(\frac 3 2-H,\alpha(\beta\wedge\theta)+\frac 1 2-H)r^{\alpha(\beta\wedge\theta)+\frac 1 2-H}\1_{[0\leq r<\delta]}\right\}.
\end{align}
Substituting \eqref{add-ineq-I41}, \eqref{add-ineq-I42}, \eqref{add-ineq-I3} and \eqref{add-ineq-I4} into \eqref{add-ineq-Th2}, we arrive at
\begin{align}\label{add-Th2-3}
&\Big|K_H^{-1}\Big(\int_0^\cdot( Z(Y_s^\xi)+h^\xi(s))\d s\Big)(r)\Big|\nonumber\\
&\qquad\le[1+C_0(H-\frac{1}{2})]\frac{r^{\frac{1}{2}-H}}{\Gamma(\frac{3}{2}-H)}
 (L_1\|\hat{\sigma}^{-1}\|\|Y^\xi\|_{0,r,\beta}\delta^\beta
+L_2\|Y^\xi\|_{0,r,\beta}^\alpha\delta^{\alpha\beta})\nonumber\\
&\qquad\qquad+\frac{r^{H-1/2}(2H-1)}{\Gamma(\frac{3}{2}-H)}\Big(L_1\|\hat{\sigma}^{-1}\|\|Y^\xi\|_{0,r,\beta}\Big[
\frac{\delta^{\beta+1/2-H}}{(2\beta+1-2H)r_\delta^{H-1/2}}+
\frac {\delta^\beta(r-r_\delta)^{\frac 1 2-H}} {(r_\delta/2)^{H-1/2}(H-1/2)}\nonumber\\
&\qquad\qquad+\frac {\delta^\beta\left(\frac {r_\delta} 2\right)^{\frac 3 2-H}} {\left(r-r_\delta/2\right)^{H+1/2}\left(\frac 3 2-H\right)}
\Big]
+ L_2\|Y^\xi\|_{-\tau,r,{(\beta\wedge\theta)}}^\alpha\Big[
\frac {\delta^{(\beta\wedge\theta)}(r-r_\delta)^{\frac 1 2-H}} {(r_\delta/2)^{H-1/2}(H-1/2)}\nonumber\\
&\qquad\qquad+\frac {\delta^{(\beta\wedge\theta)}\left(\frac {r_\delta} 2\right)^{\frac 3 2-H}} {\left(r-r_\delta/2\right)^{H+1/2}\left(\frac 3 2-H\right)}+\frac {2r_\delta^{\frac 1 2-H}\delta^{\alpha{(\beta\wedge\theta)}+\frac 1 2-H}} {2\alpha{(\beta\wedge\theta)}+1-2H}\Big]\Big)\1{[r\geq \delta]}\nonumber\\
&\qquad\qquad +C_T\left(\|\|Y^\xi\|_{0,r,\beta}r^{\beta+1-2H}+\|Y^\xi\|_{-\tau,r,{ \beta\wedge\theta }}^\alpha r^{\alpha(\beta\wedge\theta)+1-2H}\right)\1_{[0\leq r<\delta]}.
\end{align}
By the B-D-G inequality, 
\begin{align*}
\Theta_2(t)&\le C_T\Big(\E\left(\int_0^T\Big|K_H^{-1}\Big(\int_0^\cdot( Z(Y_s^\xi)+h^\xi(s))\d s\Big)(r)\Big|^2\d r\right)^{\frac {q} {(q-1)2}}\Big)^{\frac {q-1} q}\nonumber\\
&\le C_T\delta^{\alpha{(\beta\wedge\theta)}+\frac{1}{2}-H}.
\end{align*}

For $\Theta_3$, it follows from  H\"older's inequality and    \eqref{add-Th2-3} that 
\begin{align*}
\Theta_3(t)&\leq \frac 1 2\left(\E\left(\int_0^T\left|K_H^{-1}\left(\int_0^\cdot (Z(Y^\xi_s)-h^\xi(s))\d s\right)(r)\right|^2\d r\right)^{\frac {q} {q-1}}\right)^{\frac {q-1} {2q}}\\
&\qquad \times \Big(\E\left(\int_0^T\Big|K_H^{-1}\Big(\int_0^\cdot( Z(Y_s^\xi)+h^\xi(s))\d s\Big)(r)\Big|^2\d r\right)^{\frac {q} {q-1}}\Big)^{\frac {q-1} {2q}}\\
&\le C_T\delta^{\alpha{(\beta\wedge\theta)}+\frac{1}{2}-H}
\left(\E\left(\int_0^T\left|K_H^{-1}\left(\int_0^\cdot (Z(Y^\xi_s)-h^\xi(s))\d s\right)(r)\right|^2\d r\right)^{\frac {q} {q-1}}\right)^{\frac {q-1} {2q}}.
\end{align*}
Since
\beg{align*}
\int_0^T \left|K_H^{-1}\left(\int_0^\cdot (Z(Y^\xi_s)-h^\xi(s))\d s\right)(r)\right|^2\d r&\leq 2\int_0^T \left|K_H^{-1}\left(\int_0^\cdot (Z(Y^\xi_s)+h^\xi(s))\d s\right)(r)\right|^2\d r\\
&\qquad +2\int_0^T \left|K_H^{-1}\left(\int_0^\cdot  h^\xi(s)\d s\right)(r)\right|^2\d r,
\end{align*}
it follows from \eqref{add-Th2-3} and \eqref{add-ineq-hK} that 
$$\left(\E\left(\int_0^T\left|K_H^{-1}\left(\int_0^\cdot (Z(Y^\xi_s)-h^\xi(s))\d s\right)(r)\right|^2\d r\right)^{\frac {q} {q-1}}\right)^{\frac {q-1} {2q}}<\infty.$$
Hence, 
$$\Theta_3(t)\leq C_T \delta^{\alpha{(\beta\wedge\theta)}+\frac{1}{2}-H}.$$

Finally, the desired assertion is established from \eqref{eq1.12} and the estimates of $\Theta_i(t), i=1,2,3.$

\end{document}